\documentclass{article}

\usepackage{amsmath, graphicx, apacite, xcolor,comment,authblk}
\usepackage[most]{tcolorbox}
\usepackage{pgfkeys}
\usepackage{upgreek}
\usepackage{natbib}
\usepackage{hyperref}
\newcounter{nuexamples}

\newcounter{lastexample}
\input{trash.lp}
\newwrite\lastexampfile\relax
    \immediate\openout\lastexampfile=trash.lp\relax

\usepackage{ifthen}
\ifthenelse{\value{lastexample}=1}{\def\nuexamples{one\ }}{}
\ifthenelse{\value{lastexample}=2}{\def\nuexamples{two\ }}{}
\ifthenelse{\value{lastexample}=3}{\def\nuexamples{three\ }}{}
\ifthenelse{\value{lastexample}=4}{\def\nuexamples{four\ }}{}
\ifthenelse{\value{lastexample}=5}{\def\nuexamples{five\ }}{}
\ifthenelse{\value{lastexample}=6}{\def\nuexamples{six\ }}{}
\ifthenelse{\value{lastexample}=7}{\def\nuexamples{seven\ }}{}
\ifthenelse{\value{lastexample}=8}{\def\nuexamples{eight\ }}{}
\ifthenelse{\value{lastexample}=9}{\def\nuexamples{nine\ }}{}
\ifthenelse{\value{lastexample}=10}{\def\nuexamples{ten\ }}{}
\ifthenelse{\value{lastexample}=11}{\def\nuexamples{eleven\ }}{}
\ifthenelse{\value{lastexample}=12}{\def\nuexamples{twelve\ }}{}
\ifthenelse{\value{lastexample}>12}{\def\nuexamples{many\ }}{}

\usepackage[most]{tcolorbox}
\tcbset{%
  checkedstyle/.style={breakable,enhanced, sharp corners,colback={gray!10!white}}}

\def\tcbpar{\setlength{\parindent}{1.35em}%
\setlength{\parskip}{.1ex plus 0.75ex minus 0.25ex}}

\newtcbtheorem{Technical}{Technical Discussion}{checkedstyle}{Technical}

\usepackage{jr}
\newtheorem{theorem}{Theorem}[subsection]%
\newtheorem{lemma}{Lemma}[subsection]%

\theoremstyle{definition} 
\newtheorem{example}{Example}[subsection]%

\usepackage{fancyhdr}
\begin{document}

\title{From Thomas Bayes to Big Data: \\On the feasibility of being a subjective Bayesian\footnote{Supported in part by NSF Grant DMS-2113364.}}
\author[1]{Ya\hspace{-.1em}'\hspace{-.1em}acov Ritov}

\affil[1]{University of Michigan, Ann Arbor}

\maketitle

\begin{abstract}{
We argue that the Bayesian paradigm, of a prior which represents the beliefs of the statistician before observing the data, is not feasible in ultra-high dimensional models. We claim that natural priors that represent the a priori beliefs  fail in unpredictable ways under values of the parameters that cannot be honestly ignored. We do not claim that the frequentist estimators we present  cannot be mimicked by Bayesian procedures, but that these Bayesian procedures do not represent beliefs. They were created with the frequentist analysis in mind, and in most cases, they cannot represent a consistent set of   beliefs about the parameters (for example, since they depend on the loss function, the particular functional of interest,  and not only on the a priori knowledge, different priors should be used for different  analyses of the same data set). In a way, these are frequentist procedures using Bayesian technique. 

The paper presents different examples where the subjective point of view fails. It is  argued that   the arguments based on Wald's and Savage's seminal works are not relevant to the validity of the subjective Bayesian paradigm.  The discussion tries to deal with the fundamentals, but the argument is based on a firm mathematical proofs. 
}\end{abstract}


\section{Introduction and summary}

It was used to be that the Bayesian analysis was restricted to simple models. Computational considerations were an issue, and practical Bayesian analysis was limited to the simple case of exponential families with conjugate models. Thus Savage wrote: `No attempt is made to include criteria like intellectual simplicity or
facility of computation that depend not only on the estimate but also
on the capabilities of the people who contemplate using it.'' (\cite{Savage1972}, page 221) He continues: ``no one really considers
it a virtue in itself for an estimate to be a maximum-likelihood
estimate \ldots; rather, it is believed that such estimates do
typically have real virtues.''  In this note we try to apply these criteria to the Bayesian analysis itself. 

Thus, we try to consider the Bayesian approach not as a methodology to yield simple solutions, almost off-the-shelf, for complex problems. Solutions that  can be approximated in practice. It is lowering the net, since other solutions can be approximated as well, and it does not make sense to save on the final computer time for the analysis of a data which cost millions to collect. We, of course, acknowledge that Bayesian solutions may be good. One may consider as a Bayesian estimator any procedure which is the minimizer of the a posteriori mean loss for some artificial a priori distribution, or even to the maximizer of the a posteriori density. We will, however, restrict the adjective `Bayesian' to those who follow a Bayesian philosophy.

\cite{berger2000bayesian} describes the subjective Bayesian analysis as the ``soul'' of Bayesian statistics, which probably means that it is not the mind, and certainly not the reality. We admit that. However, it is hard to justify ``philosphically'' a technique. There is no ideology in M-estimation or generalized methods of moments. They do not becomes an adjective like ``Bayesian.'' Nobody refers to himself as an M-estimatian. Frequentist techniques can be judged on how close to  the  best (in any sense) solution to the problem they are. They are not the goal by itself. The division of the community is to Bayesians vs. non-Bayesians. There is no parallel division to method of moments people to non-method of moments guys.  Some may argue that, sometimes, the conclusions are not sensitive to the subjective probability, and therefore we can bypass the step of having a real subjective prior, and the philosophy stands with using practical priors. This would not be the case in the examples we are going to present. The results would heavily depend on the prior, and it would not be possible to have a good, approxiately subjective prior.

\cite{Lindley2000} outlined the Bayesian philosophy. For him uncertainty is always formalized as a probability distribution. A statistical problem has two uncertain components. Almost all statistical schools agree on the uncertainty, or randomness, of the observations given the parameters. However, statistical analysis is needed exactly since there is uncertainty about the values of the model parameters. The researcher start with his own world view or, at least, a small world view, \cite{Savage1972}. As any other uncertainty, this is also described by the so call a priori distribution.  The proper statistical analysis combines these two using the inverse probability formula, AKA, Bayes Theorem. \cite{Bernardo2011} claims that this approach succeeds in what any other approach to statistics fails, namely to give a solid axiomatic foundations that guarantees a methodological consistency, no less than ``a scientific revolution in Kuhn’s sense.''

Except for the philosophical assertions there are the mathematical arguments. They are based on two results, attributed to two giants, Wald and Savage. As Barnard wrote in his comment on \cite{Lindley1953}:  ``What Wald did from the practical point of view was to show that even if we disbelieve
in the existence of a prior distribution, we should none the less behave as if it did exist.'' Wald argued that any admissible statistical procedure is Bayesian (or, at least, a limit of Bayesian procedures). Savage argued that any utility maximizer behaves as if he has prior over the states of nature. We do not dispute these mathematical arguments (which does not mean we accept Savage's axioms). However, we will argue that they cannot justify or even motivate the positions of Lindley or Bernardo. 

The mid-twenty century  Bayesian approach described relatively simple models with only a few unknown parameters. We are now in the age of big data and deep networks. A simple current statistical model involves many unknown parameters. Regularly, there are much more unknown parameters than observations. This has a few major implications to the feasibility and success of the Bayesian approach. To begin with, it is difficult if not utterly impossible (in theory!) to have a world view about a set of thousand parameters. Suppose we, somehow, arrive to a prior which describes properly the unknown certainty of the a compound vector of parameters. If there is a consistent estimator then a Bayesian procedure is consistent under its prior, \cite{RBGK}. This is a strong statement for a procedure which is continuous in low dimension. It is a very weak statement in an ultra-high dimension world, where a set with high probability under one prior has almost zero measure by a slightly different prior. Naive assumptions about the proper world-view may bring a very misleading Bayesian procedure. Consider a simple case. Let $\Th_0\subset\R$ be a compact set. We can come with a reasonable prior describing our uncertainty on a parameter on $\Th_0$. Suppose now that we consider $\Th_0^p$, i.e., we consider a vector of $p$ parameters: $\th_i\in\Th_0$, $i=1,\dots,p$. The $\th_i$s are, to the best of our knowledge, similar and unrelated. Typically, our prior considers them independent, or at least, exchangeable. But this have a strong implications. For example, the means of any large subset of them,  $\th_{i_1},\dots,\th_{i_q}$, $i_1<\dots<i_q\le p$ are approximately equal to the expected value of each one of them (by the prior!) The Bayesian procedures are consistent with high probability under the prior, but how would they behave when the  sequence is, de facto, such that its first half has a completely different mean that its second one? This may have a meaningful effect on some Bayesian decisions. Are we really committed to these implications?

There was the replication crisis in science, and, in particular, in social psychology. One of the suggested solutions was adopting the Bayesian methodolgoy, \cite{Wasserstein29032019} and \cite{Ruberg29032019}. This was surprising. If the problem is that scientists ignore the effect of flexible stopping times (e.g., continue to collect data in an effort to stop with a P value of 5\%), the solution cannot be adopt a procedure that intentionally ignore stopping time. If ignoring the effect of too many hypotheses is a problem, then the solution cannot be a philosophy of science under which this is permissible. Finally, if you complain about lack of protocols and pre-registration, do not revoke an a posteriori analysis.   

We cannot call a Bayesian anyone who use Bayes Theorem. In fact, Bayes himself states and proves this theorem as part of his introduction, which in modern terminology, he adds for completeness since it cannot be found in this way in other sources. See more in Technical Discussion \ref{Technical:BayesAndPrice}. We reserve the term Bayesian  to those who use Bayes Theorem as the corner stone of their statistical methodology.   We distinguish between three types of Bayesians:
\begin{enumerate}
  \item The honest/pardigmatic/subjective Bayesian (S-Bayesian). The one who follows the paradigm, and his prior represents his beliefs on the parameter previous to observing the data.  \cite{Lindley1953},\cite{Lindley1990},\cite{Lindley2000}, and \cite{Bernardo2011}. 
  \item The formal/technical Bayesian (F-Bayesian). The one that has a prior and do a posteriori analysis, but his prior is the one that is usually used for such a problem, the prior that he knows how to handle, the prior that yields a good robust estimator, etc., \cite{CastiloEtAl}, \cite{berger2000bayesian}, and \cite{berger2013statistical}.
  \item The model based Bayesian (M-Bayesian). For him the prior has a frequentist interpretation, and is based on the same consideration as the model. For example, it is the way the machine whose behaviour he monitor, was constructed. It is doubted whether we should consider M-Bayesian as true Bayesian, since all frequentists are M-Bayesians. We argue that Thomas Bayes was such a Bayesian (or non-Bayesian). See   Technical Discussion \ref{BayesAndPrice}.
\end{enumerate}

For simplicity, all non-Bayesian approaches are referred to as frequentist.

One claim that we will not discuss in length is, that even if the prior is formal and does not represent a belief, the benefit of being an F-Bayesian is that he has an a posteriori distribution and hence can express the results of the experiment using probability terminology. In particular, the claim is, credible sets are conceptually simpler than confidence intervals (although the F-Bayesian struggles with tests whose frequentist interpretation  is clearer). I believe that a posteriori distribution without the strong foundation of a prior is creatio ex nihilo. If you can argue that your a posteriori distribution depends only weakly on your prior to a level that, for all practical purpose, all reasonable priors yield the same a posteriori distribution, this a posteriori without a priori distributions can be justified. Our argument is that  in the contexts we analyze and in the \nuexamples  examples we examine, this is not the case, since the good formal Bayes procedures are based on technical priors which cannot represent a belief and therefore the a posteriori distribution is not an a posteriori belief. I will not expand further  on that.   

These \nuexamples  examples are not simple, but they are much simpler than the many models in which Bayesian analysis is applied. We do not consider networks, deep or social. The correlations are relatively simple, there are no mixed models, etc. We wanted models which can be analyzed thoroughly, and in a relative transparent way. We believe that it is harder to argue for the Bayesian philosophical position in the more complex problems.
\vspace{3ex}

In Section \ref{math} we present three major mathematical claims in favor of the Bayesian approach, consistency (Section \ref{sec:consistency}), admissibility (Section \ref{sec:wald}), and rationality (\ref{sec:savage}). Then we exemplify why high dimensions pfesent real difficulty  for a true Bayesian: the space is, in some sense, empty, Section \ref{sec:empty}, the model employed should depend on the sample size, Section \ref{sec:samplesize}, a Bayesian procedure has, necessarily, the plug-in property, \ref{sec:pie}, and we should consider what did not happen \ref{sec:counterfactional}. Most of the discussion is devoted to a detailed analysis of \nuexamples  examples. Each examples is analyzed thoroughly both from frequentist considerations and from a Bayesian point of view. A well performing fequentist procedures exist. In most cases, good  F-Bayesian procedures are possible, but, we try to demonstrate, a honest S-Bayesian is going to fail.  Some of the discussion is deemed technical and is titled as Technical Note or given in the Appendix.

\vspace{1ex}
\begin{Technical}{Thomas Bayes: the first non-Bayesian.}{BayesAndPrice}
\label{BayesAndPrice}
\tcbpar 
The fundamental paper of the Bayesian paradigm is \cite{Bayes1763} as submitted by Richard Price. It is interesting to understand to what extent and in what sense Thomas Bayes was a ``Bayesian'' in the modern sense.  We certainly are not the first to ask this question, cf. \cite{GILLIES1987}, \cite{berger2000bayesian}, and \cite{Stigler1982} for different views.

It is hard to understand a seminal work written postmortem in a very early stage of the field and describe it using modern terminology. It is not necessarily the case that it has a consistent argument.  It may even doubted what exactly was written by Bayes and what was written by Price (although Price tries to be explicit about that).   

The published paper is 49 pages long. Of them, more than half---the first six pages, the last nineteen and half, and the footnotes---were written explicitly by  Price and represents his opinion, the rest are based on notes by Bayes. There is a big difference between the two writers: Price wrote an introduction and summary, discussed the implications (including theological and philosophical ones) and, in general, his text includes verbal discussion on one hand, and tedious calculations on the other. Compared to him, Bayes text is dry and purely mathematical without an introduction or summary, and, except in the Scholium, almost without any discussion.

Bayes's part starts with a very short preamble:  ``\emph{Given} the number of times in which an unknown event has happened and failed: \emph{Required} the chance that the probability of its happening in a single trial lies somewhere between any two degrees of probability that can be named.'' In modern terminology, after observing $X_n\dist B(n,\th)$, what can be said on the distribution of $\th$ and the probability the next trial in the sequence will be a success? 
It is interesting to note that the words \emph{chance} and \emph{probability} are defined explicitly as synonyms---Definition 6: ``By \emph{chance} I mean the same as probability.'' Thus, in modern terminology, the question is what is the posteriori distribution of $\th$ if we observe binomial random variable with probability $\th$ of success. 

Section 1 is   opened  with a set of seven basic definitions of probability theory. Probability is defined in Definition 5, while Definition 6 just says that {chance} has the same meaning.
Probability is  defined with respect to a real \emph{happening} of an objective event. The probability is   the ratio between its \emph{expected value} and the the actual materialized  value. Note that the ``expectation \ldots ought to be computed.''  The notion is objective and has a positive sense. It is not about subjective believes. I conjuncture that this is the reason he prefers assume that $P(X_n=k)\equiv 1/(n+1)$  as the primitive assumption---he prefers probabilities of real observed event ($X_n$) to the more abstract assumption of the probability distribution of the parameter.  Even in the case analyzed by Price (see below), where the analysis is more Bayesian than the above, the parameter has a real meaning, and the answer whether the bet on its being in a range is something that can objectively verified as it is a property of a real physical object (the lottery wheel). 

What is known today as Bayes Formula is given next as part of what Price refers to as ``a brief demonstration of the general laws of chance.'' In particular Prop. 3 (in the third page of the Bayes's part) states what can be written as $P(A\inter B)=P(A)P(B\mid A)$, while Prop. 5 (on the sixth page) claims $P(A\mid B)=P(A\inter B)/P(B)$, where $A$ is the first event and $B$ is the second event in the ordered pair $(A,B)$.

The main part of the paper is Section 2. It is the application of Bayes Formula to a specific model. As said, it considers $n$ Bernoulli trials,  $X_n\dist B(n,\th)$, and $\th$ has a standard uniform distribution.  It is important to say that there is nothing subjective about the uniform prior.  Bayes' starting point, as is written in the Scholium to the section,  was the assumption that for any $n=1,2,\dots$ and $k=0,\dots,n$ we assume $P(X_n=k)=1/(n+1)$. This assumption is mathematically equivalent to the assumption that $\th$ has a uniform distribution and this can be proved (with 18th century rigours) using the tools that are heavily used in this 1763 paper  (e.g., the law of larger numbers as applied to Bernoulli trials attributed in this essay to De Moivre. Seemingly, \cite{Stigler1982} is missing this point). Now, this is not a subjective prior and it  has a frequentist sense (a term that Bayes could not use, but considering sequences of trials with increasing to infinity length is the center of this paper, in fact, Price complains that De Movire's statements are meaningful only when the sequence length is infinite). Moreover, Bayes' prior is, in the frequentist sense, as exact as the binomial model and follows from the physical description of the described ideal experiment. 

Bayes analyzed an experiment with two identical  balls thrown similarly on a table and whose  resting points are recorded. The first ball is thrown once, the Bernoulli trials are the $n$ throws that follow, where success is recorded whenever the second ball rests on the table to the right of the resting point of the first ball. The assumption that  $P(X_n=k)=1/(n+1)$ is no more than the assumption that the first $n+1$ resting points are independent and identical  trials and the resting point of the first ball can, therefore, equally likely, be  any of the $n+1$ points. It is true that it is written in Postulate 1   that ``there shall be the same probability that it rest upon any one equal part of the plane as another and it must necessarily rest somewhere upon it.'' However, this is not necessary and is done only for simplicity---replacing areas by measures (or cumulative distribution functions), and leaving the proof verbatim otherwise,  yields the same argument without the redundant postulate about the uniform distribution of the resting points.  We want to note that Price, in his general discussion in the Appendix follows this scheme were the binomial experiment starts \emph{after} the first event, and the first event is used to define the properties of the experiment, in particular, the prior is a consequence of the observed first event.

Bayes had the tools to deal with other type of priors (although with a great difficulty)---the most important mathematical book  Bayes published in his life was \cite{Bayes1736}: \emph{``An Introduction to the Doctrine of Fluxions: And Defence of the Mathematicans Against the Objections of the Author of the Analyst''} (the Mathematicans refer to Newton while the Analyst is Berkeley) in which a detailed presentation of calculus (AKA the doctrine of fluxions) was given.  Why does he insist on the uniform prior? In his introduction to the Essay,  Price claims that Bayes originally considered a much more general than the minimal model describes in Sections 1 and 2. This is consistent with what Bayes writes in  the Scholium.  Clearly, Price thinks that Bayes was over-cautious and that he would  prefer that Bayes would consider the more Bayesian assumption, that the prior is uniform. In fact, he (Price, but attributed to a disposed draft of Bayes) thinks that it ``must be'' uniform. What is this necessity? Is it technical, because he knows how to deal with the uniform distribution, or is a methodological necessity? I think that we should be careful with the presentation of Bayes's goals. He is not Bayesian. He does not investigate how to combine a partial knowledge with new evidence. His problem is defined very precisely. The first sentence of the Essay says: ``Given the number of times in which an unknown event has happened and failed: \emph{Regquired} the chance\dots'' The only input is the Data and nothing beyond it. In the Scholium it is stated precisely that he deals with ``the case of an event concerning the probability of which we \emph{absolutely} know nothing antecedently to any trial concerning it'' [emphasis: mine].  His problem is not how to incorporate different sources of knowledge, He wants to understand the epistemology of learning from experience without \emph{any} previous information.  His ``prior'' should be neutral, not based on a belief, but the one that follows from the assumption that any result is as likely as any other. It should be an event which he calls ``an unknown event,'' which I understand in modern terms simple as using a non-informative prior.  However, he still warn: ``I may justly reason concerning it as if its probability had be first unfixed, and then determined in such a manner as to give me no reason to think that, in a certain number of trials, it should rather happen any one possible number of times than another.''  His result are relevant if and only if the underline mechanism justifies the uniform prior. The result, that unlike a true Bayesian, his conclusion is not a modified belief, but he is able to say, ``if nothing is known concerning an event but that it happened $p$ time and failed $q$ times \ldots, and from hence I guess that the probability  of its happening\ldots the chance I am in right in my guess is\ldots'' The uniform is the true distribution of the real physical parameter and the conclusion is a prediction of its true value, a prediction which is either right or wrong.

 The name Hume is not mentioned in this paper (nor any other non-mathematician), but the response to Hume seems to be one of the main motivation, \cite{GILLIES1987}. The question posed by Price is how to go beyond the law of large numbers as given by De Moivre and Simpson, and be able to argue from long but finite sequences on the laws governing the phenomena, ``what reason we have for believing that there are in the constitution of things fixt laws \dots and, therefore, the frame of the world must be the effect of the wisdom and power of an intelligent cause; and thus to confirm the argument taken from the final causes for the existence of the Deity.'' Price argues that this goal is common to the authors of the paper and De Moivre, but, Price says, the asymptotic methods of the latter do not give a convincing answer after observing a long but finite sequence.   We should note that the essay cannot be an answer to Hume. The all discussion is within the framework of repeated Bernoulli trials, and this is exactly what Hume's skeptic question is all about---how can we argue for stationarity. In fact, Price is aware that all he was able is ``to show us what conclusions to draw from \emph{uniform} experience.'' We always can suspect that there are forces in nature which may interfere with the operation of the observed constancy.

In the Appendix, Price applies an example of a polyhedron with unknown properties that is tossed millions times (if we will be extremely petty, conditioning on the first observed face is formally  wrong). We obsere the same  face over and over again, and the question is what is the probability of observing another face. He apply his tedious calculations and show that it becomes less likely but always possible that what we observe is not a deterministic law of nature.  Here he explains that that this is a simile to the human experience. ``But if we had once seen any particular effects, as the burning of wood on putting it into fire, or the failing of a stone on detaching it from continguous objects, then the conclusions to be drawn from any number of subsequent events of the same kind would be to be determined in the same manner with the conclusions just mentioned to constitution of the solid I have supposed.'' 

But then Price becomes the first Bayesian, or at least a pre-Bayesian. He is aware that, in reality, the prior is not going to be uniform but much more concentrate around $\th=1$: ``After having observed for some time the course of events it would be found that the operations of nature are in general regular,\ldots The consideration of this will cause one or a few experiments often a much stronger expectation of success\ldots than would otherwise have been reasonable.''  It can be argued that his problem is in the line of sequential empirical Bayes, \cite{Samuel1963}. 

Next, in the last part of the Appendix he considers a lottery where the ratio of \emph{Blanks} to \emph{Prizes} is observed to be 1:10. He considers different situations in which this same ratio is observed after 11, 22, 110, 1100, and 11,000.  In all five cases he is asking essentially the same question. In modern terminology the question is: what is the a posteriori distribution that the odds of the lottery are between $1:11$ to $1:9$? To be more precise, his question is whether the chance of winning in  betting on this range is larger than half, i.e., whether $P(\frac{1}{12}<\th<\frac{1}{10}\mid \text{Data})>0.5$ or not. What makes him a Bayesian is  statements like that with 110 observations ``it will remain unlikely that the true proportion should lie between 9 to 1 and 11 to 1, the chance for this being $0.2506$.'' He really takes his uniform prior seriously to claim the a posteriori with such a precision! An observer can bet safely 1:3 on the true proportion to be in this range.  However, to be a real S-Bayesian, he had to argue for the prior that justifies this bet. To the best of my understanding, he does not have an argument.

\end{Technical}
\vspace{1ex}

\section{The mathematical foundations}\label{math}

There are three mathematical major arguments for Bayesian statistics. In short, the first claim is that if there are consistent estimators, then Bayes procedures are necessarily consistent. Secondly, all admissible decisions are, at least in some sense, Bayesian. Thirdly, rational behavior (e.g., one that follows Savage's postulates) is Bayesian.  Surprisingly, although every one agrees with the mathematical arguments, many are not convinced with what seems to be  their  implications. In this section we argue that, indeed, one cannot base the Bayesian ideology on these three claims. I believe this is true in general, but it is clearly the case in the current ultra-high dimensional statistics---exactly where the current Bayesian statistics flourishes. 

\subsection{Two reference points}

When is the Bayesian analysis done? We should differentiate between two frameworks; the one  we  call the a priori (or by the protocol) Bayesian analysis,  and the other which is a posteriori (or at the obtained data) procedure. The standard statistical decision discussion start with a loss function $L(\th,d)$ which describes the loss occurred when the true value of the parameter that governs the data collection is $\th$ and an act $d$ (e.g., its guessed value, AKA the estimate). Of course statisticians are pessimistic in nature, and unlike economist who put in the front utility and gain, statistician thinks about loss and risk.  The second element is a data collection stage in which data is collected or, at least brought to the front if it was collected in the past. A statistical procedure is the decision to be done after the data was analyzed. Formally, it is a function of the data (or, more generally, the data and an external randomization. Technically, a decision procedure is a Markov kernel).   A theoretically proper frequentist analysis is actually done before the data are brought forth. It is the analysis of the what is going to be done. The orthodox frequentist is analyzing the functions that he would  apply to the data, not their eventual specific values.  After the data is investigated, the statistician should only calculate the values of the functions decided upon and reporting them. Statistical decisions are these functions. The loss occurs is, thus, necessarily, random, since the act $d$ is a function of the yet to be observed data, and thus as random as the data. If the decision function is $\delta$, the expected value of the loss, $\E_\th L(\th,\delta(X))$ is called the \emph{risk}, and is denoted by $R(\th,\del)$. Note that the second parameter of $R$ is a function and that $\th$ appears twice in the expectation:  once as an argument inside the loss function and once as the parameter under which the expectation is done.  

This is the only way that frequentist analysis can be understood. Even in the reality when the reported analysis was envisioned only after an initial analysis was done, the report is as if it was done by the protocol, considering the counterfactual of the data as a random value. Notions like consistent estimator, efficiency, admissibility,  P value, and confidence intervals have no sense unless the they apply to random data. 

The Bayesian position is more complex. The a priori Bayesian has the same consideration as the frequentist. She considers the risk function and choose the procedure which would minimize the expected loss, or the risk function. Since the Bayesian considers the parameter as a random variable, she would evaluate $\del$ by considering $\E_\uppi R(\th,\del)$, where $\uppi$ is the a priori distribution of $\th$. The a posteriori Bayesian considers the loss function and its expectation after observing the data, expectation calculated using his a posteriori distribution. Of course, it is a trivial mathematical fact that in normal circumstances there is no real difference between the two---the optimal a priori procedure is to minimize the a posteriori expectation of loss function. However, there are substantial differences between the two when it comes to the evaluations of the procedures and their positions.   The a priori Bayesian considers the sample space explicitly (you need it for calculating the risk function), also, like the a posteriori Bayesian, eventually it would be factored out. However,  the prior is essential for him, and notion like admissibility and consistency are meaningful. On the other hand,  it may seem that the a posteriori Bayesian cares more about having a posteriori distribution than the validity of the prior that gives it (e.g.,  noninformative priors, \cite{berger2000bayesian}).

\subsection{All Bayes procedures are consistent, but only on their own terms}\label{sec:consistency}

The emphasis of this paper is the ultra high dimensional settings where the dimension of the problem increases with $p$. Sometimes it is implicitly so. For example in the classical nonparametric estimation of a density of a real variable, or the nonparametric regression of one variable on another, where the latter has a  very low dimension. Other times it is explicit, e.g., the regression of a real variable on $p$ real variables based on $n$ observations, where $p>n$, and thus, necessarily, $p=p_n$ and the parameter itself depends on the sample size $n$. As a result the standard notion of consistent estimator is not clear.  To emphasis this, we  use the notion of persistent estimator introduced in \cite{GreenshteinRitov2004}. An estimator $\hat\beta_n$ of a parameter $\beta_n$ is $c_n$-persistent in distance $d_n$ if $d_n(\beta_n,\hat\beta_n)$ is of order $c_n$ as $n$ grows to infinity.   

Note that when we speak in this paper on a parameter space, an estimator, or a prior, we actually consider sequences that depend on the sample size $n$. If we consider the simple linear model, $Y=\beta\t X+\eps$, we consider parameter $\beta$ with dimension $p=p_n$ which grows (sometimes very fast) with $n$. Thus the prior is a prior on $p_n$ variables, and the estimator is based on $n$ observations and has values in $\R^{p_n}$ and thus necessarily depend on $n$. The embedding in a growing dimension is not only technical, to be able to speak about asymptotics, it may be essential: the dimension of the problem is within the eyes of the beholder, since the number of confounders is infinite---any function of the existing confounders and any iteraction of thereof is another confounder. The question is where do we stop, and stopping point is a function of the sample size and the willingness of the researcher to deal with increased complexity.

The strong result described formally in Technical Discussion \ref{Technical:const} says that the existence of any uniformly persistent estimator implies that whatever is the prior over the parameter space, the posterior distribution most likely concentrated around the true value. In that sense, the Bayesian cannot be wrong.

This is a very strong result for the M-Bayesian---the estimator is consistent where it matters. It may be strong enough for the  S-Bayesian---it is consistent for parameters he cared about. If $\scb$ is a compact parameter set that does not depend on $n$, and if the loss is uniformly continuous, this result ensure that a Bayes estimator is consistent for any parameter in its support. However, we claim that it is a very weak result in the ultra-high dimension spaces we discuss nowadays, in particular for a Bayesian whose prior does not represent knowledge or a honest belief.  A sequence of  priors is always concentrated on  small sets, and two very similar priors are asymptotically orthogonal, since there is  a sequence of sets such that the probabilities of these sets converge to 1 under one of the priors and to 0 under the other.  This implies that there are interesting sequences $\beta_n$, that cannot not honestly be rejected a priori, that the Bayes procedures would not be consistent under them. 

\vspace{1ex}
\begin{Technical}{Bayes Consistency}{const}
\tcbpar
\cite{RBGK} argued that the existence of a consistent estimator ensures that \emph{any} Bayes procedure is consistent in its own terms. To be precise, they considered a loss function $\ell_n(\beta,\hat\beta)$---a metric which depends on $n$. In the applications it is typically a scaled  norm, e.g., $n^{\al}\|\hat\beta-\beta\|_1$. We extend their result to be relevant to our setup, and  consider  a parameter space $\scb_n$ and observation $X_n\dist P_{\beta_n}$ for $\beta_n\in\scb_n$. 

An estimator $\hat\beta_n=\hat\beta_n(X_n)$ is $\ell_n$ uniformly persistent on  $\scb_n$ if under all parameter values in $\scb_n$, $\ell_n(\beta_n,\hat\beta_n)$ is uniformly bounded in probability (i.e., for all  $M_n\to\en$ and $\beta_n\in\scb_n$, $\lim_{n\to\en}P_{\beta_n}\bigl(\ell_n(\beta_n,\hat\beta_n)>M_n\bigr)=0$).  

For any $\uppi_n$, a prior distribution on $\scb_n$, let $\uppi_n^{X_n}$ be the formal a posteriori distribution, and let ${\beta^{\uppi_n}_n\mid X_n}\dist\uppi_n^{X_n}$. That is, we consider the Markov process $\beta_n\to X_n \to \beta_n^{\uppi_n}$, where  $\beta_n,\beta_n^{\uppi_n}$ have marginal distribution $\uppi_n$ and are \iid given $X_n$. The prior is $\ell_n$ uniformly persistent if  there is subset $\scb'_n\subseteq\scb_n$ with $\uppi_n(\scb'_n)\to 1$, such that for all  $M_n\to\en$ and for all $\beta_n\in\scb_n'$, $P_{\beta_n}\bigl(\ell_n(\beta_n,{\beta}^{\uppi_n}_n)>M_n)=0\bigr)$. In other words, a prior is $\ell_n$-persistent, if for most parameter values, the posteriori is likely to concentrate around that value (as measured by $\ell_n$).

Suppose the frequentist estimator is $\ti\beta_n$ is $\ell_n$ uniformly persistent. This implies tht $\ti\beta_n$ is close to the true $\beta_n$ (as measure by $\ell_n$). Now, $\ti\beta_n$ is a (random) function of $X_n$, and conditional on $X_n$, $\beta_n$ and $\beta_n^{\uppi_n}$ are i.i.d., therefore,  if $\ti\beta_n$ is close in probability to $\beta_n$ it is necessarily close to $\ti\beta_n^{\uppi_n}$ in probability. But then, $\beta_n^{\uppi_n}$ is close to $\beta_n$ by the triangle inequality, and $\uppi_n$ is persistent. Formally:

\begin{theorem}\label{th:consistency}
If there is an  $\ell_n$ uniformly persistent procedure $\ti\beta_n$ on $\scb_n$, then any prior on $\scb_n$ is $\ell_n$ uniformly persistent.   
\end{theorem}

A weaker result  was proven in \cite{RBGK}. The proof is given in Appendix \ref{app:proofs}.

\end{Technical}
\vspace{1ex}

\begin{example}\label{ex:plm1}\addtocounter{nuexamples}{1}
Consider the vector parameter $\boldsymbol{\beta}_n=(\beta_{n,1},\dots,\beta_{n,p_n})$ of some problem. To the best of our a priori knowledge the $\beta_{n,i}$s are exchangeable and $|\beta_{n,i}|\le 1$. I.e., their order has no meaning and all permutations are a priori equally likely. Any subjective prior over the parameter set have some properties. For example, if $S_p\subset {1,\dots,p}$ and $|S_p|(1-|S_p|)\to\en$ (where $|S|$ is the number of entries in the set $S$), then 
 \eqsplit{
    \frac1{|S_p|}\sum_{j\in S_p}\beta_{n,j} - \frac1{p-|S_p|}\sum_{j\not\in S_p}\beta_{n,j}
  }
is of order   $\max\{ \sqrt{1/|S_p|},\sqrt{1/(p-|S_p|)}\bigr\}$. This is a very strong claim, which ``most'' sequences do not follow. The persistency of a hones prior claimed above is valid only on the rare sequences that obeys such conditions.

We consider  the following Neyman-Scott type model as a concrete example. Our sample is all the patients from the $n$ hospitals in the country during a year. The data do not include the identity of the hospital to keep privacy, but since we consider a meaningful hospital effect, we know which patients belong to same hospital. There are only a few patients per hospital, and for simplicity we assume that we have only two (or the first two patients during the year).  We consider a mixed model. Let $(g_i,h_i)$, $i=1,\dots,n$ be some unobserved values,  The model assumes: 
 \eqsplit{
    X_{i} &= g_i + \eta_{i}& &
    \\
    Y_{i} &= h_i+\th X_{i}+\eps_{i} 
    \\
    \ti X_{i} &= g_i + \ti \eta_{i}& &
    \\
    \ti Y_{i} &= h_i+\th \ti X_{i}+\ti\eps_{i},
  }
where  $(\eps_i,\eta_i)$ and $(\ti \eps_i, \ti \eta_i)$ are  \iid normal independent of $(g_i,h_i)$. 
The regression parameter  $\th$ is the only one of interest. We assume that a priori there is no relation between $\th$ and the $g$s and $h$s.
Note that 
 \eqsplit{
    \ti Y_{i}-Y_{i}&= \th (  \ti X_{i} - X_{i})+  \ti\eps_{i}-\eps_{i}.
  }
Thus, it is very easy to estimate $\th$. The conservative frequentist can just consider the regression of   $\ti Y_{i}-Y_{i}$ on $  \ti X_{i} - X_{i}$ which are free from the nuisance parameters $g_1,\dots,g_n$ and $h_1,\dots,h_n$. The Bayesian would do something similar, which depends on her prior on $g_1,\dots,g_n$ and $h_1,\dots,h_n$, but asymptotically similar to the frequentist approach.  Assuming that the $(g_i,h_i)$ are \iid observations is quite reasonable and safe (cf. \cite{efron2019bayes,ritov2024no}). If she assumes further that they are jointly normal, the observations $(X_i,Y_i,\ti X_i,\ti Y_i)$ are jointly normal, with observations consisting of the two different means, and the ten entries of the covariance matrix of the four variables, while the parameters are $\th$, the two means of $g$ and $h$, and the three entries of each of the covariance matrices  of $(g, h)$ and of   $(\eta, \eps)$---nine in total.     

We add now a different aspect. There are $w_1,\dots,w_n$, $w_i\in\{0,1\}$, and the second patient, $(\ti X_i, \ti Y_i)$  is observed only if $w_i=1$.    The $g_i$s and $h_i$s are  related to the quality of the treatment and the characterization of the patients. The $w$s on the other hand, are related to the size of the relevant population in the hospital area.  Recall that the identity of the hospital is concealed from the data analyzer, and therefore, the S-Bayesian has no much choice than to assume a priori that the $w$s are independent of the $(g,h)$s. He must admit,  this is a description of his knowledge (AKA, ignorance). In reality, there may be a small spurious correlation between them, which he sincerely believe that is negligible.  Note that the distribution of $(X,Y)$ has five observable but still the same nine parameters---a substantial overparametrization.

The conservative frequentist insists on being unbiased, would still have no real difficulty. He proceeds as before using only the observed pairs, i.e., the sub-sample with $w_i=1$.   The Bayesian, on the other hand, have no reason to ignore the sub-sample with $w_i=0$. Since, by his prior, it is just a random sub-sample and he can use it to learn about the joint distribution of $(g,h)$, which, by his prior, is going to improve the utilization of the sub-samples without missing observations, \cite{IHmodel}.
\end{example}
\subsection{All admissible procedures can be approximated by Bayesian procedures}\label{sec:wald}

A classical claim going to Wald is that all admissible procedures can be approximated by Bayesian procedures. \cite{RBGK}  adapted  the result to the context of models with dimension growing up to infinite in a fast rate. The standard approach adopted by \cite{RBGK} is to approximate the statistical  model by a finite family of the distribution. Then, for a given loss function, it is argued that the collection of all  possible risk  functions (of randomized procedures) generates a convex set, and all admissible procedures are on its boundary. Application of the Hahn-Banach Theorem show that the admissible procedures in this approximation are Bayes, and typically unique.  We outlined this argument not in the hope that the reader who does not know it already the proof will  understand it, but, mainly,  to emphasis  what this argument use, and, in particular, what it does not consider.

In summary,  the prior suggested by this proof is a function of the utility function and the admissible procedure the frequentist statistician use. Nothing is said about the world view of the statistician. In fact, in general, we will argue, the implied prior \emph{cannot} represent a world view. The only conclusion that can be derived from this result is that if you have a statistical computer program that can implement only Bayesian procedures, you can use it to solve most statistical problems.

  Suppose we have a large sample of \iid observations, $X_1,\dots,X_n$, from a parametric family $\scp=\{P_\th:\;\th\in\Th\}$ for some $\Th\subseteq\R^p$, where $p$ is small. We then estimate $\th$ by our favorite estimator, $\hat\th$, may it the MLE or a Bayesian estimate with respect to some prior. Our object of interest may be $\th$, but it can be another object,  for example, the cumulative distribution function ($cdf$) at a particular point, the median of distribution, its mean, etc. We may even wish to estimate the density function. Typically, we will use a simple plug-in estimator, e.g. $F_{\hat\th}(x_0)$, $F_{\hat\th}^{-1}(0.5)$, $\int x f_{\hat\th}x\,dx$, or $f_{\hat\th}(\cdot)$, respectively, where $F_\th$ and $f_\th$ are the cdf and the density of $P_\th$.

  However, if we consider the similar nonparametric problem, where it is assumed only that the density belong to some smoothness class, we do not use a plug-in estimator. A frequentist will, probably, estimate the cdf, the mean, and the median using the empirical distribution function (i.e., $n^{-1}\sum_{i=1}^{n}\ind(X_i\le x_0)$, $n^{-1}\sum_{i=1}^{n}X_i$, and $\argmin_{x}n^{-1}\sum_{i=1}^{n}\ind(X_i\le x)\ge 0.5$, respectively), while the density is going to be estimated using a kernel estimate (or an equivalent statistics). If our interest is the derivative of the density, we will use, probably, another kernel estimator with a different bandwidth (if the same bandwidth is used,  then at least one of these two estimator will not achieve the optimal convergence rate).
  
  Similarly, a typical F-Bayesian considers different priors for the different targets of estimation. However, unlike the frequentist, one may expect the Bayesian to be methodologically consistent.   After all, there is only one unknown object, the probability function governing the distribution of the sample, and the prior should quantify the uncertainty of the Bayesian about this parameter. The S-Bayesian cannot do this.

\begin{example}\addtocounter{nuexamples}{1}
  We consider now a toy problem---the white noise model. It can be argued that it is isomorphic to the density estimation problem described above (e.g., using harmonic expansion of the density). However, it is conceptually simple. Suppose $X_i\dist \scn(\mu_i,n^{-1})$, $i=1,2,\dots, $, independent. Practically, as common today, their number is bounded by some  $p\gg n$. That is, at stage $n$ we observe essentially an infinite number of random constants with independent added noise. We assume that, for sure, $|\mu_i|<c_i$. We want to estimate $\th=\sum\al_i\mu_i$, for some constants $\al_i>0$ such that $\sum_{i=1}^{\en}\al_i=\en$. However,  both $\sum_i\al_i^2<\en$ and $\sum_{i=1}^{\en}|\al_ic_i|<\en$. There is nothing that tell us how the $\mu_i$s are related to each other, and, admittedly, they are a priori free, independent variables. Moreover, $\mu_1/c_1,\mu_2/c_2,\dots$ are, to the best of our knowledge, exchangeable and symmetric about 0. 
  
  The loss function is the standard quadratic loss. That is, the Bayes estimator is the expected value. Since the expectation of a sum is the sum of the expectations, $\th_n^B=\sum_{i=1}^{\en}\al_i\mu_i^B$, where $\hat\th^B,\hat\mu_1^B,\hat\mu_2^B,\dots$ are the Bayes estimator of the corresponding parameters, i.e., the a posteriori expectation. Similarly, the typical frequentist is using $\sum_{i=1}^{p}\al_iX_i$, or, maybe, $\sum_{i=1}^{p}\al_i[X_i]^{c_i}_{c_i}$, where $[x]_a^b$ is $x$ truncated to the interval $[a,b]$. 
  
  But now, consider a more difficult case: $\al_i=1/i^\xi$ and $c_i=1/i^\nu$, for $\xi<1/2$ and $\nu+\xi>1$.  The typical frequentist would use the estimator $\hat\th_n=\sum_{i=1}^{m}\al_i X_i$ for $m\approx n^{1/(2\nu-1)}$. It has a variance $n\sum_{i=1}^{m}\al_i^2$ and its bias-squared is equal to $\bigl(\sum_{i=m}^{\en}|\al_ic_i|\bigr)^2$, both of them of order $n^{-1+\frac{1-2\xi}{2\nu-1}}$. It is less than the parametric rate of $1/n$, but it can be argue that this is the minimax rate.

  The conservative S-Bayesian, who sticks to his beliefs, would use the same prior as before. Note that for $i\gg n^{1/(2\nu)}$, the prior distribution of $\mu_i$,  which necessarily is supported on $[-c_i,c_i]$ is much more informative about $\mu_i$ than $X_i$, whose error is of order $n^{-1/2}\gg c_i$. Thus, his estimator of $\mu_i$ hardly depends on the value of $X_i$. Compare it to the above mentioned typical frequentist who estimates $\th$ using $X_i$ all the way up to $n^{1/(2\nu-1)}\gg n^{1/(2\nu)}$. It may seem that the Bayesian is more sensible. 
  
  However, this is not the case. The frequentist is afraid  from the bias more than he worries about the variance. Bias accumulates when we add the weighted estimates of the $\mu_i$s. Zero mean noise, by nature, accumulate only in the mean square, since there is partial cancelation between periods when the noise is positive and periods when the noice is negative. The Bayesian by adding the terms between $n^{1/(2\nu)} $ to  $n^{1/(2\nu-1)} $ may add a bias of order $n^{-1+\frac{1-2\xi}{2\nu}}$ which is larger than the frequentist typical error. This happens since, for the Bayesian, $\mu_i$ itself is a random variables, and adding the above mentioned terms for the Bayesian is exactly like adding noise for the frequentist---the terms partially cancel each other.
  
  But, we know that there are other priors, and there are Bayes estimators which are as good as the estimator used by the frequentist. For example, suppose that other the prior, $\mu_i = \mu^*_i \al_i$, for $m^*=n^{1/(2\nu)}<i<n^{1/(2\nu-1)}=m$ and $\mu_i^*$ \iid with some prior distribution on $[-n^{-\frac{\nu-\xi}{2\nu-1}},+n^{-\frac{\nu-\xi}{2\nu-1}}]$. If this is the prior, then $T=\sum_{i=m^*}^{m}\al_i X_i$ is a sufficient statistics for $\mu$, and therefore the Bayesian is going to use the same statistics as the frequentist. Note that this prior is different from the prior discussed above, and it depends heavily on the specific functional to be estimated.
  
  I leave it to the reader what world view is represented by this prior.
  
\end{example}

\begin{example}\addtocounter{nuexamples}{1}
\label{ex:plm2}
Consider the same model as in Example \ref{ex:plm1}. There is a simple Bayesian fix. If under this prior $(g_i,h_i,\eps_i,\eta_i)$ has a  distribution that depends on the value of $w$. As a result, the sub-sample, $\{i:\;w_i=0\}$ has no information relevant to $\th$.
Most F-Bayesians feel good about such a prior. It it true that it does not express their  world-view, but it protects well against spurious correlation between $w$ and $(g,h)$.   But, this just proves our point.
\end{example}

\subsection{Savage's rationality argument}
\label{sec:savage}

\cite{Savage1972} considers a rational agent who has to choose between different actions. He has a utility which depends on his action and the unknown state of the (small) world.  The argument is that if the agent follows some (arguable) postulates, he compares the weighted average of his utility for each action over the world states. Seemingly, for each action he considers the expectation of the utility with respect to some  distribution of the state of the world. However, nothing in the argument says that this expectation expresses his beliefs about the world.  

The relevant actions for the  a priori Bayesian are the decision functions. Naturally, he considers the utility of each decision function by the expectation of the (negative of the) risk function with the respect to his prior distribution of the parameter. However, this cannot be based on \cite{Savage1972}.
In his case, a minimal condition for the rationality of a procedure is that it is admissible. Although one can think about situations in which one prefers an inadmissible procedure over a procedure that dominates it. E.g., he may consider arguments like simplicity, computational feasibility, elegance, and robustness, we may consider having them as inadequate definition of the risk function, since from the utility-theory point of view, the risk function was supposed to be corrected and the procedure should be penalized if it lacks simplicity or is numerically infeasible. Thus, we should be concerned only with admissible procedures.  But  admissible procedures are (almost) Bayesian as we already argued,  \cite{Savage1972}  does not add much information. However, we already also argued that the distribution needed for making a non-subjective Bayes procedures admissible cannot express subjective a priori understanding.

The situation of the a posteriori Bayesian is different. The time point in which she may consult \cite{Savage1972} is after observing the data, and she should choose an action, not a procedure. The Bayesian argument says that the loss function should be weighted by the a posteriori distribution, but this is more than can be found in the argument of \cite{Savage1972}. Nothing compels him to relate what he does after observing $X$ to what he would do after observing $X'$ (an important point of the a posteriori Bayesian that data that was not observed should not be considered). 

\begin{example}\label{ex:CODA}\addtocounter{nuexamples}{1}
  We try now to examplify that the rational prior probability may not describe a world-view but is defined by the best (frequentist) performance of the resulting decision procedure. We consider a simple model in which there are  $n$ unknown  parameters $p_1,\dots,p_n$, (essentially) unrelated and known sampling weights $w_1,\dots,w_n$. The sampling weights defined independent sample indicators $S_1,\dots,S_n$, $P(S_i=1)=w_i$. If $S_i=1$, we observe a Bernoulli random variable $Y_i$ with $P(Y_i=1)=p_i$. The target is to estimate the mean of the $p_i$s.
  
  There are two immediate estimators. The simpler one, and the one we claim the S-Bayesian should approximately use, is the mean of the observed $Y_i$s. The other is  the  modified Horovitz-Thompson estimator, which is semiparametrically efficient. This estimator is a weighted average of the observed $Y_i$s with weights inversely proportionally to the sampling probabilities $w_i$.  The latter estimator is asymptotically unbiased. The estimator used by the S-Bayesian will have a somewhat smaller variance, but it  fails miserably being biased if the the prior assumptions do not hold and the weights and the probabilities have some small moderate (unexplained) correlation. It can be argued that in high dimensions, the Horovitz-Thompson estimator is the rational procedure. Now, it is certainly possible to have a formal prior which generates the frequentist estimator. Some were suggested in the literature.  But these formal priors would strongly deviate from the model assumptions (which both the Bayesian and the frequentist strongly believe to be, at least approximately, valid). This estimator works, but does not represent subjective beliefs about the value of the unknown parameters.  Thus, we argue, in this model, Bayesian computations can be used but the Bayesian formulation lack the methodological foundations. 
  
  The actual argument is tedious. See Appendix \ref{ex:CODAdetails}  for the necessary details of this example.
\end{example}

\section{What can get wrong in ultra-high dimensions }

There are a few things that can get wrong and which fails the Bayesian paradigm, even when the formal Bayesian computations can survive.
\subsection{High dimension space are empty}\label{sec:empty}

My imagination is failing me when the dimension is greater than 4. For example, I imagine a uniform distribution in the hundred dimensions unit ball as some extension of the way I perceive random samples in   three dimensions. But then the following facts are paradoxical. Suppose we have one million \iid observations uniformly distributed inside a 100 dimensions ball with radius 1. We could imagine that the points will be in all the ball from the center outward. However,  the distance of a random observation from the center  has the cumulative distribution function $r^{100}$. Thus, 99\% of the sample are in the outer thin shell of  width of approximately 0.045. In fact, with high probability, none of the one million observations will be within the inner ball of radius 0.87. More strickenly, the points are completely isolated and are very far one from another---most likely, the minimal distance between two random points out of the $0.5\times10^{12}$ pairs is higher than 0.94. For a given random point of the million in the unit ball, it is likely that the rest are at a distance which is between 1 and 1.7.         

This is simple, but it has many implications. The way Savage's and Wald's results are argued is based on finite approximation of the parameter space. If the dimension of the parameter space is ultra-high, where most current statistical discussion is, it is impossible to approximate the distributions with reasonably small number of points  without assuming very strong conditions. 

Another implications is that most of the mass of two distributions over the ultra-high dimension space may be on distjoints sets, although the two distributions are perceived similar. For obvious example, much of the elementary statistical analysis are that it is possible to test whether   $\th_1,\dots,\th_p$ are \iid standard normal or \iid Gaussian with mean 0.002 and variance 1 if $p>250,000$ (i.e., the number of pixels of $500\times500$ picture). But this means exactly that the behavior of a sample from $\scn(0,1)$ and $\scn(0.002,1)$ may be dramatically different in an important aspect. It is rare that subjective considerations can distinguish a priori between such similar pairs of distributions.  

This means that it is almost meaningless to say  that if an estimator is consistent in the sense that under ``most'' values of the parameter then it behaves properly. Even if the estimator is behaves properly under the prior which was used in its construction, it may behave poorly under a very similar prior.

The next subsections will exemplified the difficulty of having  prior on high dimensions. Many times, a reasonable prior (reasonable in that it generates reasonable estimators), would have to concentrate on a particular subset of the space. For example, it would concentrate near a particular low dimensional manifold, which is defined not by the subjective beliefs, but the particular application of the desired estimator. 

\subsection{The ``prior'' should depend on the sample size $n$}
\label{sec:samplesize}

One of the main implicationss of the Bayesian approach is the strong likelihood principle, \cite{birnbaum1962foundations}, \cite{berger1988likelihood}, and \cite{Mayo2014}. A conclusion is that the Bayesian procedure does not depend on unobserved data. Only on the prior and the  likelihood function of the actual observed data.  We consider in most cases a setup, in which we have a random sub-sample from a big data set, which itself is a sample from an ideal distribution.   It follows from the above logic that the prior cannot depend on the sample size. This works well for the simple models considered in  the twentieth century. It is less clear that this works in general when the models are complex with complexity that may grow with $n$. When we have a large data set, we typically are more ambitious and fit a more complex model to the data. It is not clear how it can be argued consistently that since you gave me one million observations and not just the thousand I hoped to get, I changed my uncertainties about the model parameters.

\
\begin{example}\label{thisexample}\addtocounter{nuexamples}{1}
  The bread and butter of the modern statistical analysis is the ultra-high dimensional linear model: $Y_i=\beta\t X_i+\eps_i$, $i=1,\dots,n$, where $\eps_i$ is independent of $X_i=(X_{i1},\dots,X_{ip})\t$, $p\gg n$. The dimension of $p$ can be as high as almost being exponential in $n$, so also the theoretical asymptotic is, necessarily, with $p=p_n$, the dependency is vague as long as $p_n\gg n$, and in the discussion below, we would consider  $p$ in some sense as fixed: $n$ as changing, but in the range where it is much smaller than $p$ and $p$ is hardly changing.  The goal is estimating $\beta$, either by itself, with the loss function $L(\beta,\hat\beta)=\|\hat\beta-\beta\|^2=\sum_{j=1}^{p}(\hat\beta_j-\beta_j)^2$, or as a predictor of $Y$: $L(\beta,\hat\beta)=n^{-1}\sum_{i=1}^{n}(\hat Y_i-Y_i)^2$, where $\hat Y_i=\hat\beta\t X_i$.   To simplify the discussion further, we may consider a toy version of it that retain the essense of the general model:  
   \eqsplit[uhdr]{
    \text{For every $n$:  }W_{nj} &= \beta_j+ n^{-1/2} \eps_j, \qquad j=1,\dots,p,\quad \eps_1,\eps_2,\dots  \text{ \iid\ }\scn(0,1).
    }
   The target now is estimating $\beta_1,\dots,\beta_p$ under the same quadratic loss function: $L(\beta,\hat \beta)=\sum_{j=1}^{p}(\hat\beta_j-\beta_j)^2$.  We assume that $\sum_{j=1}^{p}\beta_j^2<\en$
   
   Since $p\gg n$,   a simple fit of the model to the data yields nonsense. In the regression model the model will fit the noisy data exactly, and the error in the toy version would be $p/n\to\en$. The current standard approach  has two elements. First, a priori (in the Bayesian or frequentist sense) all the $\beta$ coefficients are equivalent. We do not have a reason to assume that one of them is larger than the other. But, the second point is, we have a reasonably good, low dimension sub-model. I.e., we can tactically assume that most coefficients are zero. The question is which of them, since a priori any small subset could be the best small set of explanatory variables. 
   
   There is a standard estimating procedure under stronger than necessary conditions on the design matrix, $n^{-1}\sum_{i=1}^{n}X_iX_i\t$, namely the lasso, suggested first by \cite{Tibshirani1996} and analyzed first in the context of ultra-high dimension analysis in \cite{GreenshteinRitov2004}. In its penalized version:
    \eqsplit{
        \hat\beta_L &= \argmin_{\ti\beta} \{\sum_{i=1}^{n}(Y_i-\ti\beta\t X_i)^2+\lambda \sum_{j=1}^{p}|\ti\beta_j|\},
     }
    where $|\beta|=\sum_{i=1}^{p}|\beta_i|$. This may look like a MAP (maximum a posteriori) estimator, as mention by \cite{Tibshirani1996} and argued in \cite{Gribonval}. Formally, indeed, it is the maximum of the posterior if the model is \eqref{uhdr} with $\eps_i\dist\scn(0,\sig^2)$, i.i.d.,  independent of the $X_i$s, and under the prior,  the $\beta_j$s  are \iid with  double-exponential distribution whose  mean is $\sig^2/\lambda$. There is a debate whether the MAP is a Bayesian procedure, it certainly does not follow the utility maximizer point-of-view for continuous parameter space. A standard Bayesian analysis was suggested by \cite{CastiloEtAl}. Usually it is done using a spike-and-slab prior. Consider the toy version for simplicity. The lasso solution for the toy version is
     \eqsplit[tve]{
        \hat\beta_L &= \argmin_{\ti\beta} \{\sum_{i=1}^{n}(W_{j}-\ti\beta_j)^2+\lambda \sum_{j=1}^{p}|\ti\beta_j|\}
        \\
        \implies\quad \hat\beta_{Lj} &= \argmin_{\ti\beta_j} \{(W_{j}-\ti\beta_j)^2+\lambda |\ti\beta_j|\}
        \\
        &= \begin{cases}
             W_j-\lm/2, & \;\;W_j>\lm/2 \\
             0& |W_j|\le\lm/2\\
             W_j +\lm/2 & \;\;W_j<-\lm/2.
           \end{cases} 
      }
      The spike-and-slab prior is a prior that considers the $\beta_j$s as i.i.d., with probability $1-\gamma$ that $\beta_j=0$, and, by the prior, if $\beta_j\ne 0$, then it is taken from some distribution, say, $\scn(0,\tau^2)$. The a posteriori distribution is that 
       \eqsplit{
        \uppi(\beta_j=0\mid W_j) &= \frac{(1-\gamma)\frac{1}{\sqrt{2\pi}}e^{-nW_j^2/2}} {(1-\gamma)\frac{1}{\sqrt{2\pi}}e^{-nW_j^2/2}+ \gamma \frac{1}{\sqrt{2\uppi(n^{-1}+\tau^2)}} e^{-W_j^2/2(n^{-1} +\tau^2)} } 
        \\[2ex]
        &= \frac{1}{1+ \frac{\gamma}{1-\gamma}  \frac{1}{\sqrt{n^{-1}+\tau^2}}  e^{\frac{nW_j^2}{2}  \frac{n\tau^2}{1+n\tau^2} }},
        \\[2ex]
        \uppi(\beta_j\mid W_j) &=   \frac{\frac{\gamma}{2\pi\sqrt{\tau^2/n}} e^{-\beta_j^2/2\tau^2}  e^{ -n(\beta_j-W_j)^2/2} }   {(1-\gamma)\frac{1}{\sqrt{2\pi}}e^{-nW_j^2/2}+ \gamma \frac{1}{\sqrt{2\uppi(n^{-1}+\tau^2)}} e^{-W_j^2/2(n^{-1} +\tau^2)}} 
        \\[2ex]
        &=  \uppi(\beta_j\ne 0\mid W_j) \sqrt{\frac{1+n\tau^2}{\tau^2}}  e^{-\frac12(n+\tau^{-2})(\beta_j - \frac{n\tau^2}{n\tau^2+1}W_j )^2}.
        }
     Thus, the a posteriori distribution has a point mass at zero, and otherwise it has $\scn(\frac{n\tau^2}{1+n\tau^2}W_j,   \frac{\tau^2}{1+n\tau^2}  )$.  If $\gamma$ and $\tau$ were fixed in $n$, the sample size is large relative to the a priori slab variance and most of the $\beta_j$s are 0: 
      \eqsplit{
        \uppi(\beta_j=0\mid W_j) &\approx \frac{1}{1+\gamma e^{nW_j^2/2}}
        \\
        \uppi(\beta_j\mid W_j) &\approx \frac{1}{1+\gamma e^{-nW_j^2/2}} \sqrt{n}e^{-n(\beta_j-W_j)^2/2}. 
       }
       The Bayesian estimator is a smooth version of the lasso estimator of \eqref{tve} with $\lm\approx \sqrt{2|\log\gamma|/n} $.
     
     Now, we don't say that the Bayesian techniques are useless as a computation tools. We do not argue against a statistician that find a good estimator and then a prior which generates it. We try to argue that doing this is not within the S-Bayesian paradigm that the prior expresses prior knowledge. A particular demand from a priori statement (i.e., a statement that is done before the sample is known) is that parameters of the prior would not depend on the sample size. In particular $\gamma$ and $\tau^2$ should be fixed. In fact, if we assume that $\sum_{j=1}^{p}\beta_j^2\approx c$, then $\gamma p\tau^2\approx c$. In particular, $\gamma p$, the predicted number of non-zero terms does not grow with $n$. This is in contrast to the standard assumptions in the sparse model estimation, that the number of non-zero terms is growing with $n$. For example, if $\beta_{(1)},\dots, \beta_{(p)}$ are $\beta_1,\dots,\beta$ ordered from the largest to the smallest absolute value, then a standard assumption is that $\beta_{(j)}=0$ for $j>n/\log(n)$, or $|\beta_{(j)}|<j^{\al}$ for some $\al>1/2$. In the latter case, we consider the models with $M=n^{1/(1+2\al)} $ non-zero terms. The frequentist is not oblige to any prior, his model is tentative and maximize what can be learn from the data. For example, \cite{GreenshteinRitov2004, RigolletTsybakov2012} consider the best model with $M=M_n$ non-zeros coefficients. If we consider $\beta_{(j)}\approx j^{-\al}$, the optimal model size to fit the toy version minimizes the bias contribution of a small model, $\approx \sum_{k=M}^{p}k^{-2\al}\approx M^{-2\al+1}$ and the estimation error, $M/n$, thus,  $M_n\approx n^{1/(2\al)}$.  Thus the optimal value of  $\lm$ in \eqref{tve} satisfies     
      \eqsplit{
        \frac{n^{1/(2\al)}}{p} &=  2\bigl(1- \Phi(\sqrt{n}\lm/2)\bigr) \approx \frac{4}{\sqrt{2\pi n}\lm}e^{-n\lm^2/8},
        \intertext{implying}
        \lm &\approx   \sqrt{\frac{8\bigl(\log( p)-\log(n)/2\al\bigr)}{n}}.
       }
     The dependency of $\lm$ on $n$ is weak, but is necessary for having a good estimator. In practice, the fine tuning of $\lm$ can be done using a cross-validation-like procedure. 
     
     The frequentist may consider, as above, that $j^{\al}|\beta_{(j)}|$ is bounded. The $\beta_j$ are exchangeable, can they be considered as a simple sample? That is, can we find a distribution such that its empirical distribution is $\bbp_n(\xi) \approx 1 - c/(p\xi^{1/\al}) $? Not directly. Of course, one can generate $\ti\beta_j \dist \scn(0,j^{-2\al})$, $j=1,2,\dots,p$, i.i.d.,  and then take a random permutation of them.  The spike-and-slab prior tries and fails to generate something like that. The non-zero terms are too big and their numbers does not grow with $n$. The prior of \cite{CastiloEtAl} has two independent components. The number of non-zero terms and their distributions. The prior of the number of components  depends on $p$, which, by itself is problematic. More problematic, is that the prior is constructed such that it would give a reasonably large number of non-zero terms as a posterior, while by the prior itself, their number is actually pretty small, and is bounded with high probability at a bound that depends on $p$.  See Technical Note \ref{Technical:uhdModelSelection}.
     
     \vspace{1ex}
     \begin{Technical}{A non-Bayesian prior }{uhdModelSelection}
     \tcbpar
     A Bayesian analysis of the ultra-high dimension regression model is presented in  \cite{CastiloEtAl}. We are going to discuss their prior and whether it is relevant to the Bayesian paradigm. They consider a prior of the form:
      \eqsplit{
        \pi(\beta) &= \pi(|\scm|)\frac{1}{\binom{p}{|\scm(\beta)|}}\pi(\beta_{\scm(\beta)})\pi(\beta_{\scm(\beta)^\perp})  
       }
       Where $\scm(\beta)\subseteq\{1,\dots,p\}$ is the set of the significant non-zero coordinates of $\beta$, $|\scm(\beta)|$ is its cardinality, and $\beta_{\scm }$ are the value of the non-zero terms. Thus, the prior  considers the process of selecting the size of the model, then uniformly one of the subsets of $\{1,\dots,p\}$ of this size, and then the value of these coordinates.
       
       The problems with this prior start with the model size. Their main assumptions is that model size is sub-exponential: 
       \eqsplit{
            \pi(|\scm|=m)\le C_1\Bigl(\frac{C_2}{p^{\al}}\Bigr)^m. 
       }
       We have two difficulties with this prior. The first is that the prior depends implicitly on $n$ through $p$. But, since we consider $p$ as almost fixed, we do not see as crucial. More bothering is the fact that the probability decreases too fast to 0, and, by the prior, the model size is not likely to be larger than its a priori minimum. The  purpose  is applying this super-sparse prior to a model that although is sparse, it is not too sparse, $|\scm(\beta)|\approx n^{\al}$, say, probably even $n/\log(n)$. The reason is that the the number of models of size $m$ increases super-exponentially fast in $m$ and the prior is set to suppress this growth and prevent fitting a too large and unrelated model to the data.  The prior does not describe  the statistician's beliefs about the model.\\
       
       The prior of the coordinates' values is having them \iid double exponential with scale parameter $\sqrt n/p\le \lm\le c\sqrt{n\log p}$. The upper bound is that used by the lasso which set its value  to ensure the sparsity of the model. This is not needed now, since the sparsity of the model is ensured directly by the other component of the prior. In fact, it is not obvious why $\lm$ cannot set to 0. Again, we claim that the suggested prior does not express the beliefs of the researcher. Consider the case that the sparsity the researcher is going to assume in his (e.g., frequentist) estimation is  $\scm(\beta)\approx n^{\al}$.  If he  considers $\lm=\lm_0$ which does not depend on $n$, he essentially believes that $\sum_{j=1}^{p}\beta_j^2 = 2\lm_0^{-2} n^{\al}\bigl(1\pm 5n^{-\al/2} \bigr)$. However, he also believe that $\var (Y)$, and therefore $\sum_{j=1}^{p}\beta_j^2$ is of order 1. Consider empirical Bayes approach or hierarchical prior will not solve the difficulty as long as it is assumed that $\lm_0$ is of order 1 and it is unknown constant that does not depend on $n$.   On the other hand, if he takes $\lm=\lm_0 n^{\al/2}$,  his prior depends on the sample size and thus does not express an a priori belief which was valid before the sample was taken.
       
       There is another difficulty in the standard analysis. It is based on the compatibility assumption of \cite{buhlmann2011statistics}. A condition like that, or any of the similar conditions in the literature is necessary for to establish a good behavior of the lasso and for the estimation of $\beta$. However, it is an unverifiable condition and it is not necessary for an estimation of the prediction vector $(\beta\t X_1,\dots,\beta\t X_n)\t$. See Technical Discussion \ref{Technical:wocompat}. 
     \end{Technical}
     \vspace{1ex}

     \begin{Technical}{Prediction Without Compatibility}{wocompat}
     \tcbpar
     It is interesting to consider the ultra-dimension linear regression model but with no restrictions on the empirical correlation between the different confounders, as done in  \cite{RigolletTsybakov2012}:
      \eqsplit{
        Y_i &= \beta\t X_i+\eps_i,\quad \eps_i\dist \scn(0,\sig^2). 
       }
      Since they do not have any restriction on the confounders, the target of the estimation cannot be $\beta$ itself, but only the prediction error, $\sum_{i=1}^{n}(Y_i-\hat Y_i)^2$, where $\hat Y_i = \beta\t X_i$. For each sparse $\beta$ we consider the Stein unbiased estimator of the risk: 
       \eqsplit{
        R_{un}(\beta) &= \sum_{i=1}^{n}(Y_i-\beta\t X_i)^2 + 2\sig^2 |\scm(\beta)|- \sig^2,  
        }
        where we used the same sparsity notation as in Technical Discussion \ref{Technical:uhdModelSelection}. They considered Bayes procedures with respect to a prior $\uppi$:  
         \eqsplit{
            \hat Y &= \sum_\beta \frac{e^{-\gamma R_{ub}(\beta)}\uppi(\beta)}{\sum_{\ti\beta}e^{-\gamma R_{ub}(\ti\beta)}\uppi(\ti\beta)}X\beta, 
          }
        where, as usual $\hat Y=(\hat Y_i,\dots,\hat Y_p)\t$ and $X=(X_1\t,\dots,X_n\t)\t$.  The prior $\uppi$ is similar to the prior considered in Technical Discussion \ref{Technical:uhdModelSelection}. It suffers from the inadequacies discussed there. But even if the prior is proper, this is not really a Bayes procedure, because of the value of $\gamma$. Their result depends on  $\gamma$ being not larger than $1/(4\sigma^2)$, half the value needed to make the factor before the prior being proportional to the likelihood.  Thus, this cannot be considered even as a formal Bayesian procedure. 
        
         From a Bayesian point-of-view the compatibility assumption has a major impact on the feasibility of a good prior on $\beta$ for the prediction problem.  The natural assumption is that $\boldsymbol{X}\beta$ is in a ball, where $\boldsymbol{X}$ is the design matrix. With the compatibility this translates to $\beta$ concentrated in compact sets. However, without compatibility, the assumption that $\boldsymbol{X}\beta$ is concentrated does not imply anything about the a priori distribution of $\beta$. 
         The difficult situation is when the effect is close to the level of detection: $\|\boldsymbol{X}\beta\|^2=2\sig^2\scm(\beta)\log(p)$. The vast amount of wrong models, of order $(pe/\scm(\beta))^{\scm(\beta)}$ balances their lower likelihood, and therefore the log-likelihood should be scaled up (by $\gamma$). 
     \end{Technical}
  \vspace{1ex}
   
\end{example}

\subsection{Plug-in estimation}\label{sec:pie}

When considering the estimation of a linear functional of an unknown function under quadratic loss function, the Bayesian estimator is a plug-in type: $\E L(f)=L(\E f)$. This means that the S-Bayesian does not have the option of estimating directly $L(f)$ without estimating $f$, and does not have the option of estimating $f$ differently for different functionals.

\begin{example}\label{ex:intF2}\addtocounter{nuexamples}{1}
Consider estimating $\th\eqdef\int f^2(x)dx$ where $f$ is the common density of $n$ \iid observations. For simplicity, we assume that $f\in\scf_\al$, where $\scf_\al$ is the set of all densities  supported on the unit interval,  bounded away from 0 and can be rough but there is $c<\en$ such that for any $x,x+h\in(0,1)$: $|f(x+h)-f(x)|<c|h|^{\al}$. We consider the case $1/4<\al<1/2$. It was a surprise that $\th$ can be estimated on the $\sqrt n$ rate. After all, $f^2(x)$ at a given point $x$ can be estimated at a nonparametric rate, when the bias and variance are balanced, and thus, a plug-in estimator of any reasonable estimator is not expected to yield the parametric rate. The semiparametric analysis of this model can be found, for example, in \cite{RitovBickel1990,BickelRitov1988,CaiLow2005,donoho1991geometrizing,  birge1995estimation,EfromovichLow1996}. For completeness we give an analysis of a simple estimator.

The efficient procedure can be described schematically as follows. We start with an initial estimator $\ti f(x)$ of the density. It can be based on a sub-sample. We use the expension
 \eqsplit[sq0]{
    \int f^2(x)dx &=   2\int \ti f(x) f(x)dx  -\int \ti f^2(x)dx + \int \bigl(f(x)-\ti f(x)\bigr)^2 dx 
  }
The first term is estimated easily by $2n^{-1}\sum_{i=1}^{n}\ti f(X_i)$. The second term can be calculated easily, and, hopefully, the third term is small and can be ignored. Here is a detailed simple construction. The asymptotic distribution of the efficient estimator is given in \eqref{sqeff} below. We start with dividing the sample into two equal random sub-samples and consider the histogram  estimators of the density, $\ti f_1(x),\ti f_2(x)$ with $M=M_n\to\en$ bins with equal width, where
 \eqsplit[sqRange]{
    n^{1/4\al}\ll M_n\ll n.
  }
  
Denote the value of the histograms by $\ti f_{jm}$, $j=1,2$, $m=1,\dots,M$  It is then considers 
 \eqsplit{
    \hat\th &\eqdef \frac{1}{M}\sum_{m=1}^{M} (2\ti f_{1m}\ti f_{2m}-\frac12\ti f_{1m}^2-\frac12\ti f_{2m}^2) + \frac{M}{n}.
  }   
 To justify the estimator write $\ti f_{jm}=\bar f_m+\del_{jm}$ where $\bar f_m$ is the mean true density in the $m$th interval. Then  
  \eqsplit[sq1]{
    \hat\th &= \frac{1}{M} \sum_{m=1}^{M}\bar f_j^2 + \frac1M\sum_{m=1}^{M}\bar f_j(\del_{1m}+\del_{2m}) + \frac{1}{M}\sum_{m=1}^{M}(2\del_{1m}\del_{2m}-
    \frac12\del_{1m}^2 -\frac12\del_{2m}^2) - \frac M{2n}
   }
  The first term in \eqref{sq1} is deterministic, let $\bar f$ be the histogram approximation of the density
   \eqsplit[sq1a]{
    \frac{1}{M} \sum_{m=1}^{M}\bar f_j^2 &= \int \bar f^2(x) dx 
    \\
    &= \int f^2(x)dx - 2\int \bar f(x)\bigl( f(x)-\bar f(x)\bigr)dx -\int \bigl(f(x)-\bar f(x)\bigr)^2dx  
       \\
    &= \int f^2(x)dx  -\int \bigl(f(x)-\bar f(x)\bigr)^2dx 
    \\
    &= \th+O({M^{-2\al}})=\th+ o(n^{-1/2}).  
 }
The $\del_{jm}$ are mean 0, asymptotically uncorrelated, and 
\eqsplit[sq2]{
     \frac{ n}{2M\bar f_j} \E \del_{jm}^2 \to 1.
     }
      Thus, the second term of \eqref{sq1} satisfies:     
 \eqsplit[sq1b]{
     \frac{\sqrt{n}}M\sum_{m=1}^{M}\bar f_j(\del_{1m}+\del_{2m})=\frac{\sqrt{n}}M\sum_{m=1}^{M}(\bar f_j-\th)(\del_{1m}+\del_{2m})\cid N\Bigl(0, 4\var\bigl(f(X)\bigr)\Bigr).
  }
  Finally, we consider the third and fourth terms in \eqref{sq1}. We use the fact that $M^{-1}\sum_{m=1}^{M}\bar f_m=1$, and write:
   \eqsplit[sq1c]{
    \frac{1}{M}\sum_{m=1}^{M}\bigl(2\del_{1m}\del_{2m}-
    \frac12\del_{1m}^2 -\frac12\del_{2m}^2-\bar f_j \frac M{2n}\bigr) = \OP({\frac {M^{1/2}}{n}}) =\op({n^{-1/2}}) ,
    }
   since by \eqref{sq2}, this is a sum of asymptotically uncorrelated and mean zero terms, where the terms are $\OP({M/n})$. It follows that from \eqref{sq1}, \eqref{sq1a}, \eqref{sq1b}, and \eqref{sq1c}:     
    \eqsplit[sqeff]{
        \sqrt{n}(\hat\th- \th)\cid N\Bigl(0, 4\var\bigl(f(X)\bigr)\Bigr),
     }
   which is the semiparametric efficiency bound, \cite{RitovBickel1990}.
     
It is interesting to note that the permitted range of $M_n$, as given in \eqref{sqRange}, does not include the bandwidth of an estimator with minimax rate of converges of the density at all, $n^{1/(1+2\al)}$---the frequentist has no difficulties with having a different estimator for estimating the density and for estimating a functional of it.

The (S-)Bayesian has no much of choice. For him, $f(x)$ is a random variable, and under quadratic loss function he should compute 
 \eqsplit{
    \E \int f^2(x)dx &= \int \bigl(\E f(x)\bigr)^2 dx + \int \var\bigl(f(x)\bigr)dx.
  }
 Thus, she should compute her density estimate $\E f(x)$ and inflate it by its variance estimate.  Note that if $\al>1/2$, the a posteriori variance of a good prior is $\op({n^{-1/2}})$, and thus the correction is negligible.  However, we are interested in the regime were $1/4<\al<1/2$, and the correction is essential. In fact, the minimax estimator of $f$ has mean square error (MSE) of order  $n^{-2\al/(1+2\al)}\gg n^{-1/2}$. Thus, $\int \bigl(\E f(x)\bigr)^2dx$ has estimation error that is of a bigger order than that of the frequentist, so the exact correction is essential.

But the integral is a linear operator and hence the Bayesian has the plug-in problem: he has to use $\hat\th = \int \widehat{\{f^2(x)\}}dx$. But, any good estimator of $f^2(x)$ is similar to a kernel estimator with bandwidth $\approx n^{-1/(1+2\al)}$, and has too much error. The frequentist overcomes this difficulty by introducing the term $M/n$ which indeed increases the error, but makes the estimator unbiased.  This option does not directly exist within the Bayesian paradigm. Whenever there is a statistical error, the Bayes estimator is necessarily biased, even under the assumed distribution. See Technical Note \ref{Technical:biasedBayes}.
 
 \vspace{1ex}
 \begin{Technical}{Bayes estimators are biased}{biasedBayes}
 \tcbpar
 \begin{theorem}
   Suppose $(\th,X)$ have some joint distribution (i.e., the marginal of $\th$ is the prior and the distribution of $X$ given $\th$ is the statistical model). Let the Bayes estimator under quadratic loss  be $\hat\th(X)\eqdef \E(\th\mid X)$. Then either $\hat\th=\th$ w.p. 1 or $\E(\hat \th \mid \th)\ne \th$.
 \end{theorem}
 \begin{proof}
   Since $\hat\th = \E(\th\mid X,\hat\th)$, we have also $ \E(\th\mid\hat\th)=\E\bigl(\E(\th\mid X,\hat\th)\mid \hat\th\bigr)=\hat\th$. Thus,  $\E\hat\th =\E\th$. Assume, wlog, that $\E\hat\th=\E\th=0$. Then:
    \eqsplit{
        \var(\th) &= \E\bigl(\E(\th\mid\hat\th)\bigr)^2 + \E\var(\th\mid\hat\th)
        \\
        &=\E {\hat\th}^2+ \E\var(\th\mid\hat\th)
        \\
        &=\var(\hat\th)+\E\var(\th\mid\hat\th)>\var(\hat\th), 
     }
  unless $\E\var(\th\mid\hat\th)=0$ w.p. 1. If $\hat\th$ was an unbiased estimator, $\E(\hat\th\mid\th)=\th$, then we had similarly  that $\var(\hat\th)>\var(\th)$ which is a contradiction.  
 \end{proof}
  \end{Technical}
\vspace{1ex}

We have the following result. 
\begin{theorem}\label{th:densitySq}
  Let $\th(\gamma,w)\eqdef\int f^{2\gamma}(x)w(x)dx$ be a parameter of interest.  For any $f_0(x)\in\scf_\al$ there is a set $\scs$, a symmetric band around $f_0$, such that the rate of convergence of the Bayesian estimator is worse than $n^{\al/(1+2\al)+\eps}$ for any $\eps>0$ and for any prior on \scs that does not depend on the specific values of $\gamma$ and $w$.  
\end{theorem}

The proof can be found in Appendix \ref{app:proofs}. The set $\scs$ is built in the following way. The  $[0,1]$ interval is divided into $K$ bins, and the the density is of the form $f_0(x)+\sum_{k=0}^{K-1}\xi_k g(x-k/K)$, where $\xi_i\in\{-1,1\}$ and $g$ is a given function supported on $[0,1/K]$. The difficulty is that any $\xi_i$ is estimated using only an order $n/K$ observations found in the $i$tn bin, and thus  the estimation of $\xi_i$ is necessarily biased. The plug-in property enforces plugging in these bias estimates into a convex function, and thus accumulating bias.

Can the Bayesian modified his prior and achieve the semiparametric bound? Of course he can. However, the needed formal prior would reflect his silent acceptance of the semiparametric analysis of the problem and not his a priori uncertainties. The prior would be formal to yield a good estimator, and would depend on the parameter estimated. Thus he would use different priors for estimating $\int f^{2\gamma}(x)w(x)dx$ depending on $\gamma$ and $w(\cdot)$---priors that depend on the functional and ignore his subjective beliefs on the underlined model \scs.  

\vspace{1ex}   
\begin{Technical}{Two close points}{twopoints}
\tcbpar
\begin{lemma}\label{lem:twopoints}
Consider a two points prior $\uppi( M=\mu)=\uppi( M=-\mu)=1/2$. Suppose $X\mid M\dist \scn( M,\sig^2)$, where $\mu\ll\sig$. Then
 \eqsplit{
    \uppi( M=\mu\mid X) &= \frac12+\frac{\mu}{2\sig^2}X+\OP({\frac{\mu^2}{\sig^2}}). 
  }
  
\end{lemma} 
\begin{proof}
    \eqsplit{
        \uppi( M=\mu\mid X) &=\frac{\frac12 e^{-\frac{(X-\mu)^2}{2\sig^2}}}{\frac12 e^{-\frac{(X-\mu)^2}{2\sig^2}}+\frac12 e^{-\frac{(X+\mu)^2}{2\sig^2}}}
        \\
        &=\frac{e^{\frac{X\mu}{\sig^2}}}{e^{\frac{X\mu}{\sig^2}}+e^{\frac{-X\mu}{\sig^2}}}
        \\
        &=\frac12+ \frac{X\mu}{2\sig^2} + \OP({\frac{\mu^2}{\sig^2}}). 
     }
     
\end{proof}
\end{Technical}  
\vspace{1ex}

\end{example}
\subsection{What didn't happen}\label{sec:counterfactional}

One of the Bayesian claim is, we should not care about data that was not observed. However, they do care about parameters which are not relevant to the truth.

\begin{example}[Inference without compatibility]\addtocounter{nuexamples}{1}
  We consider the ultra-high dimension version of the partial linear model, which is the estimation of a low dimensional subset of the regression coefficients. This model was introduced and analyzed from a standard frequentist point of view as discussed by different authors, \cite{ZhangZhang, Zuerich, JavanmardMontanari,CaiGuo} with the compatibility or alike assumption, and in \cite{LawRitovInference} without any of these restrictions on the design matrix. The standard model is 
   \eqsplit[iwc0]{
    Y&= \beta X+ \gamma\t W + \eps, 
    }
  where $\eps$ is independent of $(X,W)$, $W,\gamma\in\R^{p}$, $\beta,X\in\R$, and $\beta$ is the parameter of interest. This model is equivalent to the model:  
   \eqsplit[iwc]{
    X &= \phi\t W+\zeta,
    \\
    Y&=  \psi\t W+\xi,
    \\
    \beta &= \frac{\cov(\zeta,\xi)}{\var(\zeta)},    
    }
   where $(\zeta,\xi)$ are independent of $W$  and $\psi=\gamma+\beta\phi$. We note that the matrix $\boldsymbol{W}=(W_1,\dots,W_n)\t$ is typically ill-defined.  The model introduces a few challenges to the Bayesian. 
   
   \cite{LawRitovInference} considered the following estimator. For a given (large) $m$:
    \eqsplit{
        \boldsymbol{\hat\zeta} &=  \displaystyle\sum_{|\scm|=m} \frac{\exp(-\frac1{\alpha}\|\boldsymbol{X}-\boldsymbol{W}\phi_\scm\|^2)} {\sum_{|\ti\scm|=m}\exp(-\frac1{\alpha}\|\boldsymbol{X}-\boldsymbol{W}\ti\phi_{\ti\scm}\|^2)} \bigl(\boldsymbol{X}-\boldsymbol{W}\phi_\scm\bigr), 
        \\
        \boldsymbol{\hat\xi} &=  \displaystyle\sum_{|\scm|=m} \frac{\exp(-\frac1{\alpha}\|\boldsymbol{Y}-\boldsymbol{W}\psi_\scm\|^2)} {\sum_{|\ti\scm|=m}\exp(-\frac1{\alpha}\|\boldsymbol{Y}-\boldsymbol{W}\psi_{\ti\scm}\|^2)}
        \bigl(\boldsymbol{Y}-\boldsymbol{W}\phi_\scm\bigr),
        \\
        \hat\beta &=\frac{\sum_{i=1}^{n}\hat\zeta_i\hat\xi_i} {\sum_{i=1}^{n}\hat\zeta_i^2},
     }
    where bold denotes the vectorized version of the sample, e.g., $\boldsymbol {X}=(X_1,\dots,X_n)\t$ and $\boldsymbol{W}=(W_1\t,\dots,W_n\t)\t$. The set  $\scm\subseteq\{1,\dots,p\}$, and  $\phi_\scm$ and $\psi_\scm$ are the least square estimate corresponding to the models where all coefficients not in $\scm$ are set to zero.

    In words, for each model of size $m$, we fit a simple linear regressions of $Y$ annd $X$ on $W$, and weighted them inversely proportional to their residual sum-of-squares.  \cite{LawRitovInference} proved that under some regularity conditions the scaled estimation error, $\sqrt{n}(\hat\beta-\beta)$,     is asymptotically normal under some regularity conditions. The two main assumptions are, firstly,  the linear model \eqref{iwc} can be fitted with a sparse model of cardinality no larger than $\sqrt{n}/\log(p)$, and secondly, $\al>4\max\bigl(\var(\zeta),\var(\xi)\bigr)$. The latter condition is not consistent with a simple Bayesian interpretation, as \cite{LawRitovInference} point out.

    Sparse models, which are the bread and butter of the current big data statistics, are very problematic. A major difficulty is that the tail of the regression coefficient may be tiny, much below any detection (of order as small $1/\sqrt{p}$ while $p\gg n$) and yet explain totally the dependent variable. Since it is below detection, using the results of any statistical procedure as a description of the reality is based totally on an unverifiable belief.  The cautious frequentist has no real problem. The claim isn't that the estimator is approximating the theoretical best predictor, but only that it approximates the best predictor based on a model of a size $m$. Cf., \cite{GreenshteinRitov2004, RigolletTsybakov2012}, while the underlined assumption is only that there is a sparse model that give a better than nothing predictor.  When the target is estimating a single coefficient (or, more generally, a few of them) like $\beta$  in \eqref{iwc0} or \eqref{iwc} the issue is more profound. First, the aim is to achieve the $\sqrt{n}$ rate, which put a further limit on the largest model that can be dealt by (to no more than $\sqrt{n}/\log(p)$ or so while the limit is $n/\log(p)$ in the nonparametric problem). Second, the scientific goal of inference, to test whether $X$ has a direct impact on $Y$ controlling for $W$ is not achieved. At most, one can honestly claim that $X$ has direct effect after linearly controlling for any very small subset of the confounders. Third, the scientist may hope to be able to make a statement about causality, but this is not really possible. The S-Bayesian has the further challenges. For example, his prior expresses a priori beliefs and hence should be the same whether the task is estimating $\gamma$ or $\beta$ in the model \eqref{iwc0}, although, for example, the optimal frequentist sparsity assumptions must be different.
    
    Consider the simple case where the coordiantes of $W$ are statistically independent. The needed sparsity assumption are strong. Let:
     \eqsplit{
        \rho_\psi(m) &= \inf_{|\scm|<m}\E(Y-\psi_\scm\t W), \quad m=1,\dots,p,\quad \rho_\psi^c(m)=\rho_\psi(p)-\rho_\psi(m) 
        \\
        \rho_\phi(m) &= \inf_{|\scm|<m}\E(X-\phi_\scm\t W), \quad m=1,\dots,p,\quad \rho_\phi^c(m)=\rho_\phi(p)-\rho_\phi(m) 
      }
    When a Bayesian is facing the model \eqref{iwc}, he should believe in two things. First that the $X$ and $Y$ can be explained by sparse model, and the sparsity should be strictly smaller than $\sqrt{n}$ (otherwise, the estimation error of the two residuals would prevent estimating $\beta$ on the $\sqrt{n}$ rate). Second, that the residuals from these two regression when consider as Euclidean vectors are each of order smaller than $n^{-1/4}$ (and thus, the correlation between the residuals will be in order smaller than $n^{-1/2}$). It makes sense that he may believe that $X$ can be explained by $W$ in a way that satisfies these two conditions. It may sense that he believes that the regression of $Y$ on $W$ has these properties. Suppose he has no reason to assume that there is any relation between the way that $W$ explains $X$ and the way it explains $Y$. This is what he believes. We cannot argue about that. Unfortunately, this is not enough. Suppose $X$ and $Y$ can be explained with models of size $m=n^{-1/3}$. Even if the component of $W$ are statistically independent, the coefficients of the model can be all of order $m^{-1/2}$. However, suppose these models do not overlap. The coefficients of $\psi$ corresponding to those in the active variable in the model of size $m$ of $\phi$ are all of order $n^{-7/12} $, much below the detection level. There contribution to correlation between the estimated residuals is $m\times m^{-1/2}\times n^{-7/12}=n^{-5/12}\gg n^{-1/2}$.    
    
    Can the S-Bayesian ``correct'' his prior? As in most examples, the answer is yes, but with the price of stop following the S-Bayesian paradigm. If you consider the sampling procedure, evaluate the statistical properties of different estimators, consider which estimator is consistent and robust, and then look for a prior that generates this prior, you may be a good statistician, but you are not a Bayesian.

\end{example}

\vspace{2ex}

\appendix
\vspace{2ex}
\noindent\textbf{\Large Appendix}
\section{A detailed analyss of Example \ref{ex:CODA}}\label{ex:CODAdetails}

\textbf{Example \ref{ex:CODA} (details).} We try now to examplify that the rational prior probability does not describes world-view but is defined by the best (frequentist) performance of the resulting decision procedure. 
  
  We start with a simple model. There are $n$ independent normal random variables: $Y_i\dist \scn(p_i,1)$. The number $n$ is very large. The means are unconnected and exchangeable. The first target is to estimate the $p_i$s. There are no much options here. The Bayesian statistician has a prior on $p_i$. The $p_i$ are independent under the prior and, therefore, under the posterior distribution, and the estimator would be $\hat p_i=\hat p(Y_i)$, $i=1,\dots,n$. For the frequentist, this is the standard compound decision  (or empirical Bayes) problem. The efficient frequentist considers the $p_i$ \emph{as if} they come from \iid sample (it follows a mathematical argument that this is the right thing to do and the decision maker's beliefs are irrelevant.). He would then used the estimator $\hat p_i=\hat p(Y_i\mid Y_1,\dots,Y_n)$, where the estimation procedure starts with the estimation of the Bayes procedure relative to the (estimated) empirical distribution of $p_1,\dots,p_n$,  cf. \cite{efron2019bayes, SahaGunt20,ritov2024no}.   This would puts the naive Bayesian in a clear disadvantage relative to the empirical Bayesian (who is not a Bayesian!). A way out for the S-Bayesian is to admit that by posing an exact prior on the $p_i$ he goes beyond his actual knowledge, and to use a hierarchial prior which considers the $p_i$ as \iid sampled from $\uppi_p$, and put his real prior on $\uppi_p$. 
  
  In the rest of this example we consider Bernoulli random variables $Y_i\dist \scb(1,p_i)$, $i=1,\dots,n$. Since sample from a mixture of Bernoulli random variables looks exactly like a sample from a Bernoulli distribution,  the Bernoulli model is simple for the frequentist, and he had very few real options. It raises, however, some difficulties for the Bayesian. For simplicity, let's assume that under his prior, each $p_i$ has a beta distribution with shape parameters $\alpha$ and $\beta$. Under the hierarchial model the Bayesian have a prior on $\alpha$ and $\beta$, cf. \cite{li2010bayesian}. For more simplicity, we write $\alpha=\tau\rho$ and $\beta=\tau(1-\rho)$, and assume some prior distribution on $(\tau,\rho)$. The joint distribution of the parameters and the observations is therefore:
   \eqsplit{
    \propto \uppi_\rho(\rho)\uppi_{\tau\mid \rho}(\tau\mid\rho) \prod_{i=1}^{n} \Biggl[\frac{\Gamma\bigl(\tau\rho\bigr)\Gamma\bigl(\tau(1-\rho)\bigr)} {\Gamma(t)} p_i^{\tau\rho-1+ Y_i}(1-p_i)^{\tau(1-\rho)-Y_i}\Biggr].
    }
   That is,  the model defines the joint distribution of the $p_i$ and the $Y_i$, as if the unobserved random effects $p_i$ are sampled from $\uppi_{p\mid \tau,\rho}$, and given $p_1,\dots,p_n$, $Y_1,\dots,Y_n$ are independent Bernoulli, $Y_i\dist \scb(1,p_i)$. The Bayesian has a prior $\uppi_\rho\uppi_{\tau\mid \rho}$ on the parameters of the beta distribution from which the $p_i$s are sampled.    Integrating out the unobserved $p_i$s, we obtain that the a posteriori distribution
    \eqsplit{
        \propto \uppi_\rho(\rho)\uppi_{\tau\mid \rho}(\tau\mid\rho)\rho^{\sum_{i=1}^{n}Y_i}(1-\rho)^{n-\sum_{i=1}^{n}Y_i} \approx \uppi_\rho(\hat p) \uppi_{\tau\mid \rho}(\tau\mid \hat p) e^{- \frac{(\rho-\hat p)^2}{2\hat p(1-\hat p)}} ,
     }
  where $\hat p=n^{-1}\sum_{i=1}^{n}Y_i $. This holds under some mild regularity conditions, which   we are mute about. Thus, as $n$ grows, the a posteriori of $\rho$ concentrate around $\hat p$, while $\tau$ remains as it was by the prior (conditional on $\rho$ having the value $\bar p$). Since, anyhow, we don't learn about the spread parameter $\tau$, we can assume (as we actually started) as having the value $\tau_0$. Thus, asymptotically, the Bayesian behave as if the $p_i$s are \iid from the beta distribution $Be(\tau_0\bar p, \tau_0(1-\bar p))$. The Bayes estimator is asymptotically:
   \eqsplit{
    \hat p_{bi} &= \hat p+\frac{Y_i-\hat p}{\tau_0+1}. 
    }

  But, actually we are not interested in estimating $p_1,\dots,p_n$ per se, but only their average $\bar p=n^{-1}\sum_{i=1}^{n}p_i$. This is a very simple model, and the frequentist estimator is the trivial unbiased $n^{-1}\sum_{i=1}^{n}Y_i$. This is, for example, the MLE in the empirical Bayes Gaussian problem, cf;, \cite{GreenshteinRitov22}.  The Bayesian will have consistent estimator, essentially only if he has a hierarchial prior (e.g., as above), and then he will work slightly  harder, as discussed above, but will get asymptotically equivalent estimator as $\sum_{i=1}^{n}\hat p_{bi}=\hat p$ . 
  
  The problem is more complicated when we have missing data which is in the focus of this discussion.  \cite{CODA} introduced the following example, which we now considerably extend. See its discussion in \cite{datta2025inverseprobabilityweightingsurvey}, \cite{harmeling2007bayesian}, \cite{wasserman1998asymptotic}, and \cite{wasserman2004all}.  The setup is as above: there are $n$ Bernoulli observations and the parameter of interest is $\bar p=n^{-1}\sum_{i=1}^{n}p_i$.   However, after the model was described and we summarized our model uncertainties, we are told that many of the $Y_i$s are missing: there are  some strictly positive (known and given) probabilities $w_1,\dots,w_n$ and Bernoulli  $S_i\dist \scb(1,w_i)$, $i=1,\dots,n$, independent and independent of $Y_1,\dots,Y_n$. We are told the value of the sampling weights,  and then we observe $(S_1,S_1Y_1), \dots, (S_n,S_nY_n)$, i.e., we observe $Y_i\dist \scb(1,p_i)$ if, and only if, $S_i=1$. The sampling weights $w_1,\dots,w_n$ were defined by technical considerations by a sampler who does not know $p_1,\dots,p_n$, the parameters of Nature. In fact, they are defined for a completely different survey with a complex design and they were chosen to minimize the sampling cost. of this other survery. 
  
  This example may look similar to Example \ref{ex:plm1}, however, there is an essential difference. In Example \ref{ex:plm1}, the parameter was defined locally (the parameter has the same value in every locality). In the current example, there is a single global parameter. Thus, in Example \ref{ex:plm1} the statistical issue was which sets to ignore. On the other hand, a frequentist can handle Example \ref{ex:CODA} only if the observations are missing completely at random (MCAR), which is the case when $w_1,\dots,w_n$ are known. As we will argue, the current problem is an easy one for the frequentist, but is almost impossible to the S-Bayesian (and need a difficult flexibility for any Bayesian).
  
  If it is really an MCAR situation, i.e., $w_i\equiv w$, a constant. then it is an easy situation, and then all statisticians will use, asymptotically, an estimator which is equivalent to the mean of the observed $Y_i$s. We start our discussion of the general case with the a frequentist analysis, to be followed by the Bayesian approach.
  
  Our frequentist will use the H\'{a}jek  modification of the estimator  given in \cite{HT1952}:  
   \eqsplit{
    \hat p_{HT} &=  \frac{\displaystyle\sum_{i=1}^{n}\frac{S_i}{w_i}Y_i}{\displaystyle\sum_{i=1}^{n}\frac{S_i}{w_i}}. 
    }
  The motivation is clear: the expectation of the numerator (the original Horovitz-Thompson estimator) is $\sum_{i=1}^{n}\frac{w_i}{w_i}p_i=\sum_{i=1}^{n}p_i$. Similarly, the expectation of the denominator is $n$. Dividing by an estimate of a constant positively correlated with the numerator is done in order to reduce the asymptotic variance.  Note that the division of the original Horowitz-Thompson estimator by an estimate of a constant made an unbiased estimator into a slightly biased one, which is only asymptotically unbiased. Moreover, we argue below, the modification makes the estimator semiparametrically efficient. The effective sample size is $n\bar w$, where $\bar w = \frac1n \sum_{i=1}^{n}w_i$.  Using Slutsky's Theorem 
   \eqsplit[HTinfluence]{
    \sqrt{n\bar w}\bigl(\hat p_{HT}-\bar p\bigr) &=  \frac{\displaystyle\sqrt{\frac{\bar w}{ n}}\sum_{i=1}^{n}\frac{S_i}{w_i}(Y_i-\bar p)} {\displaystyle\frac1n\sum_{i=1}^{n}\frac{S_i}{w_i}}
    \\
    &=\sqrt{\frac{\bar w}{ n} }\sum_{i=1}^{n}\frac{S_i}{w_i}(Y_i-\bar p)+R_n
    \\
    &= \sqrt{\frac{\bar w}{ n} }\sum_{i=1}^{n}\frac{S_i}{w_i}(Y_i-p_i) + \sqrt{\frac{\bar w}{ n} } \sum_{i=1}^{n}\frac{S_i}{w_i}(p_i-\bar p)+R_n
    \\
    &= \sqrt{\frac{\bar w}{ n} } \sum_{i=1}^{n}\frac{S_i}{w_i}(Y_i-p_i) + \sqrt{\frac{\bar w}{ n} } \sum_{i=1}^{n}\Bigl(\frac{S_i}{w_i}-1\Bigl)(p_i-\bar p)+R_n,
    }
  where $R_n\cip 0$. The main terms on the right hand side are sums of independent mean 0 random variables. Thus $\sqrt{n\bar w}\bigl(\hat p_{HT}-\bar p\bigr)$ is asymptotically mean 0 normal with asymptotic variance: 
   \eqsplit{
    V_{HT} &= \frac1{n}\sum_{i=1}^{n}\frac{\bar w}{w_i}p_i(1-p_i)+ \frac1{n}\sum_{i=1}^{n}\frac{\bar w(1-w_i)}{w_i}(p_i-\bar p)^2,  
    }

    A standard semiparametric analysis of this model  motivates the modified Horovitz-Thompson estimator. The analysis starts with considering $(w,p)$ as having some arbitrary distribution, $h(w)\uppi(p\mid w)\in\scp$, and consider the efficient estimation at the assumed sub-model $\scp_0$ where $\uppi(p\mid w)=\uppi(p)$. The full log-likelihood is 
      \eqsplit{
        \ell(h,\uppi\mid S,Y,w)=\log \bigl(h(w)\bigr)+ S\log w + (1-S)\log(1-w)+ S\log \int p^{SY}(1-p)^{S(1-Y)}\uppi(p\mid w)dp.
       }
       The influence function of the Horovitz-Thompson estimator can be inferred from the expansion \eqref{HTinfluence}: 
        \eqsplit[HT-IF]{
            \dot\ell^{*}(S,Y,w;h,\uppi) &= \frac{S}{w}(Y-\epsilon), 
         }
        where $\epsilon=\int p\uppi(p)dp$. We know that this estimator is an unbiased estimator of $\int\int p \uppi(p\mid w)h(w)dpdw$ under any distribution in  $\scp$. To prove that it is the influence function of the semiparametric efficient estimator at $\scp_0$ we should show that it can be obtained as a derivative of the log-likelihood at $\scp_0$. Thus we consider a parametric sub-model $\uppi_t(p\mid w)=\uppi(p)+ta(p,w)$, $|t|<\zeta$, where $\int a(p,w)dp\equiv 0$. (Formally, we need $a_t(p,w)$ with $\lim_{t\to 0}a_t\to a$ and $\uppi+t a_t>0,$ but this needs only a standard technical modification.) The derivative is
         \eqsplit{
            \dot\ell(S,Y,w;h,\uppi) &= \frac{SY}{\epsilon}\int p a(p,w)dp + \frac{S(1-Y)}{1-\epsilon}\int(1-p)a(p,w)dp
            \\
            &= \frac{SY}{\epsilon}\int p a(p,w)dp - \frac{S(1-Y)}{1-\epsilon}\int p a(p,w)dp
            \\
            &=\frac{\int pa(p,w)dp}{\epsilon(1-\epsilon)}S(Y-\epsilon),   
          }
          which is the influence function \eqref{HT-IF} for example if 
           \eqsplit{
            a(p,w) &= \frac{\epsilon(1-\epsilon)}{w\int (p-\epsilon)^2\uppi(p)dp}(p-\epsilon)\uppi(p). 
            }

  The estimator of the S-Bayesian depends only on the prior and on the likelihood. In particular it does not depend on the design, i.e., it does not depend on which $Y_i$s are observed, and certainly not on the design mechanism, $w_1,\dots,w_n$. Formally, since the $w_i$s are a priori non-informative independent of $p_1,\dots,p_n$, and the under the model, $Y_i$ is independent of $S_i$, the a posteriori does not depend on $w_1,\dots,w_n$, and asymptotically the Bayes estimator is going to be approximately the mean of the observed $Y_i$s (see the argument for the full observed variables), or equivalent to     
   \eqsplit{
    \hat p_b &= \frac{\sum_{i=1}^{n}S_iY_i}{\sum_{i=1}^{n}S_i}. 
    }
    We have then:
     \eqsplit{
        \sqrt{n\bar w}\bigl(\hat p_b - \bar p\bigr) &= \frac{\sqrt{\frac{\bar w}{ n} }\sum_{i=1}^{n}S_i(Y_i-\bar p)}{\frac1n \sum_{i=1}^{n}S_i}
        \\
        &=  \sqrt{\frac{\bar w}{ n} }\sum_{i=1}^{n}\frac{S_i}{\bar w}(Y_i-\bar p) + R_n
        \\
        &= \sqrt{\frac{\bar w}{ n} }\sum_{i=1}^{n}\frac{S_i}{\bar w}(Y_i-p_i) + \sqrt{\frac{\bar w}{ n} }\sum_{i=1}^{n}\frac{S_i}{\bar w}(p_i-\bar p_w) +\sqrt{\frac{\bar w}{ n} }\sum_{i=1}^{n}\frac{S_i}{\bar w}(\bar p_w-\bar p) + R_n 
        \\      
      &= \sqrt{\frac{\bar w}{ n} }\sum_{i=1}^{n}\frac{S_i}{\bar w}(Y_i-p_i) +\sqrt{\frac{\bar w}{ n} }\sum_{i=1}^{n}\frac{S_i-w_i}{\bar w}(p_i-\bar p_w) + \sqrt{n\bar w}(\bar p_w-\bar p) + R_n' 
      }
   where $R_n,R_n'\cip 0$.  The estimator is centered around the weighted mean of the  $p_i$s: $\bar p_w=\sum_{i=1}^{n}w_ip_i/\sum_{i=1}^{n}w_i$.   The first two terms are asymptotically normal with mean zero and variance (on the $\sqrt{n\bar w}$ scale):
   \eqsplit{
    V_b &=  \frac1{n} \sum_{i=1}^{n} \frac{w_i}{\bar w}p_i(1-p_i) + \frac1{n}\sum_{i=1}^{n}\frac{w_i(1-w_i)}{\bar w}(p_i-p_w)^2, 
    }
    which is somewhat smaller than $V_{HT}$.   The frequentist was ready to pay this increase in asymptotic variance to avoid the third term which is approximately equal to 
     \eqsplit{
        e_b &= \sqrt{n\bar w}(\bar p_w - \bar p)
        \\
        &= \sqrt{\frac{\bar w}{ n} }\sum_{i=1}^{n} \bigl(\frac{w_i}{\bar w}-1\bigl)(p_i-\bar p).  
      }
     By the a priori assumption that $w_i$ and $p_i$ are independent random variables, hence $e_b$ has mean zero and and its variance  is of $\O({{\var(w)\var(p)}})$,  compensating the smaller variance of the two terms. However, the situation is much worse if there is a small deviation from the a priori assumption.  
     
    The Bayesian could go to the extreme, as suggest by Ghosh, of assuming nothing on the dependency of $p_i$ on $w_i$. Suppose that the sampling probability get only a finite number of different values, $\omega_1,\dots,\omega_K$. This defines strata, and the extreme Bayesian consider a prior which enables the mean of $\{p_i: w_i=\omega_k\}$ to be essentially a free parameter, \cite{datta2025inverseprobabilityweightingsurvey}. As we already discussed, if  the $w_i$s do not vary, the estimator of the mean of the $p_i$s is asymptotically equivalent to the mean of $Y_i$s.    Applying this to any stratum with $w_i$ equal one of the $\omega_k$ we obtain the estimator
     \eqsplit[bayesWeights]{
        \hat p_{be} &=\frac1n \sum_{k=1}^{K}n_k \frac{\sum_{w_i=\omega_k}S_iY_i}{\sum_{w_i=\omega_k}S_i} \approx \frac1n\sum_{k=1}^{K}n_k \frac{\sum_{w_i=\omega_k}S_iY_i}{n_k\omega_k}=\frac1n\sum_{i=1}^{n} \frac{S_i}{w_i}Y_i,
      }
     where $n_k$ is the number of observations in the stratum $\{i:\;w_i=\omega_k\}$. We obtain that this Bayesian's estimator is equivalent to the Horovitz-Thompson estimator.
      
     Recall, we do not argue that it is impossible to find good F-Bayesian procedures. They do exist. In general, any good frequentist procedure has an equivalent good F-Bayesian procedure. Our question is whether they are expressing beliefs. Wheterh they are reasonable within the Bayesian methodology (or philosophy), not only within the Bayesian techniques or computation.  In the problem we consider, we accept that not knowing the exact prior distribution of the $p_i$s and putting a hierarchial prior on it is reasonable. However, assuming that the $p_i$s are not independent is not legitimate way out for a S-Bayesian.     
     
     Perhaps, there is a weak empirical correlation between $w$ and $p$. Probably the two are related to some \emph{unknown} latent variable. In that case, $\hat p_{HT}$ is asymptotically unbiased, while  $\hat p_b$ is asymptotically biased. It is still true that to the best of our knowledge, the two sequences are unrelated, but, as the Bayesian reasonably introduced the prior on the parameters of the the imagined a priori distribution of the $p_i$, he can introduced a potential weak correlation between the $p_i$ and $w_i$, as if the sampling weight $w_i$ carries some information (in some unknown way) about the success rate of $Y_i$. It would be legitimate to the Bayesian to modify, slightly, his prior after getting the information about the $w_i$. The a posteriori distribution would still be of a posteriori independent variables. Assuming a weak correlation may be reasonable for a honest Bayesian, assuming unrestricted correlation is not---it is not what he believes to be true. 
     
     However, under quadratic loss function, the Bayesian must use a plug-in estimator and estimate $n^{-1}\sum_{i=1}^{n}p_i$ by ideally $n^{-1}\sum_{i=1}^{n} \E(p_i\mid S_i, S_iY_i, w_i, \uppi_p)$. If $S_i=0$ this is just the prior mean. Consider now an observation with $S_i=1$. The unbiased estimator plug-in the estimator $\bar w Y_i/w_i$. Can a Bayes estimator do the same? Yes, if the distribution of  $p_i$  depends dramatically on $w_i$. A weak dependency which can be justified is not enough. There is no way a S-Bayesian implements correlation as strong as technically needed. To be more precise, consider the joint distribution of $(w,p,S,Y)$, where, $Y\in\{0,1\}$ and $P(Y=1\mid p, w,S)=p$. Then:
      \eqsplit{
        \E(p\mid Y=1,S=1, w) &= \E(p\mid Y=1, w) 
        \\
        &= \frac{\E(p^2\mid w)}{\E(p\mid w)} 
        \\
        \E(p\mid Y=0,S=1, w) &=\E(p\mid Y=0, w) 
        \\
        &= \frac{\E(p(1-p)\mid w)}{\E(1-p\mid w)}.
       }
       Hence
        \eqsplit{
            \E(p\mid Y=1,S=1, w)  - \E(p\mid Y=0,S=1, w)  &= \frac{\var(p\mid w)}{\E(p\mid w)\bigl(1-\E(p\mid w)\bigr)}. 
         }
     This, for the S-Bayesian, cannot vary as much as $\bar w/w_i$ if $w_i$ changes considerably while the dependency of $p$ on $w$ is ``weak.''  The reason that the Bayesian put the weights $1/w_i$ in \eqref{bayesWeights}, is that the estimates of the average of the $p_i$ on the sub-sample $\scs_{k1}=\{i: w_i=\omega_k,S_i=1\}$  is used as the estimate of $\E(p_i\mid w_i=\omega_k, S_i=0)$ on the sub-sample $\scs_{k0}=\{i: w_i=\omega_k,S_i=0\}$. For the small a  $w_i$ $\scs_{k0}$ is much larger than $\scs_{k1}$. But, this estimates is based on sample size of order $n_k\omega_k$, which of order $(n_k/n) (w_i/\bar w)$ smaller. Now, we consider the case where $K=K_n\to\infty$, $\omega_1>\omega_2>\dots>\omega_{K}$, and $\omega_{K}/\bar w\to 0$. In fact, we consider the non-stratified case where $w_i$ is drawn from a continuous distribution.   Thus, if the dependency of $p_i$ on $w_i$ is weak (of order $(n\bar w)^{-1/2}$), the estimate $\hat p_{HT}$ is actually closer to $\E(p_i\mid w_i=\omega_k, S_i=0)$ then the estimate  ${\sum_{w_i=\omega_k}S_iY_i}/{\sum_{w_i=\omega_k}S_i}$. Thus, the estimate $\hat p_{be}$ cannot represent the subjective prior of a S-Bayesian who takes seriously the claim that to the best of his knowledge the $p_i$s and $w_i$s are unrelated (or at most, weakly related in an unknown way).

     There are two crucial  differences between the semiparametric approach and the Bayesian. The first is that the semiparametric frequentist needs only a dependency on the $n^{-1/2}$ scale to justify her modification of the naive estimator and paying the price of increasing variance. The second, is that this is a minimax approach, which permits her to concentrate on a specific non-intuitive direction.

     Is it reasonable for the Bayesian to robustify his procedure? Is it rationale to use a prior that yields an approximation of the Horovitz-Thompson estimator? The answers to these questions may be affirmative. Adopting this position is consistent with Savage but  would deviate from the Bayesian paradigm of priors that represent beliefs. It is being Bayesian for convenience, as a technique for generating estimator, for computational convenience, or to fit a zeitgeist.     My conclusion is that even if the argument given by Savage is relevant, the alternatives should be weighted but not by anything that resembles a honest subjective probability.


\section{Proofs}\label{app:proofs}

\begin{proof}[Proof of Theorem \ref{th:consistency}] 

Consider the joint distribution $\bbp_n$ of  $(\beta_n,\ti\beta_n,\beta_n^{\uppi_n})$ under the Bayesian assumption $\beta_n\dist \uppi_n$.  

Since $\ti\beta_n$ is uniformly consistent. Then, for any $M_n\to\en$
 \eqsplit[npf1]{
    \bbp_n\bigl(\ell_n(\beta_n,\ti\beta_n)>M_n\bigr)\to 0.
  }
But $(\beta_n,\ti\beta_n)$ have the same distribution has $(\beta_n^{\uppi_n},\ti\beta_n)$, hence \eqref{npf1} implies
 \eqsplit[npf2]{
    \bbp_n\bigl(\ell_n(\beta_n^{\uppi_n},\ti\beta_n)>M_n\bigr)\to 0.
  }
The triangular inequality, \eqref{npf1}, and \eqref{npf2} imply  
 \eqsplit{
    \bbp_n\bigl(\ell_n(\beta_n^{\uppi_n},\beta_n)>M_n\bigr) &\le \bbp_n\bigl(\ell_n(\beta_n^{\uppi_n},\ti\beta_n)>M_n\bigr)+\bbp_n\bigl(\ell_n(\beta_n,\ti\beta_n)>M_n\bigr) \to 0.
  }
  
\end{proof}

\begin{proof}[Proof of Theorem \ref{th:densitySq}]

We consider $\scs$ defined by 
 \eqsplit{
    f(x) &= f_0(x) + \sum_{k=1}^{K}\frac{\xi_k}{K^{\al}}\frac{\pi}{2}\sin\bigl(\pi (Kx-k)\bigr)\ind(k-1<Kx\le k),
  }
where $f_0$  is a known density function, $K=K_n$, and $\xi_k\in\{-1,1\}$ are some unknown arbitrary permuted constants. All that is known about $\xi_1,\dots,\xi_K$ is that they are $\{-1,1\}$ and thus, necessarily since $f$ is a density, exactly half of them are $+1$. Since, nothing otherwise is known, the only subjective  prior is that all permutations of the $\xi_k$s are equally likely. Since the distribution of the observations within the bins are fully known, the only informative part in the data are the histogram counts $(N_1,\dots,N_k)$, which have a multinomial  distribution with probabilities $f_k\pm K^{-\al} $, where $f_k=\int_{(k-1)/K}^{k/K}f_0(x)dx$. The standard deviation of $N_k/n$ is of order $(nK)^{-1/2}$ while the difference in probabilities between $\xi_k=\pm 1$ is of order $K^{-(1+\al)}$. Thus the effect size is of order $\sqrt{n/K^{1+2\al}}\to 0$. 
 \eqsplit{
    \frac{N_k}{n}\dist \scb\Big(n,(f_K\pm K^{-\al})/K\Bigl)
  }
 Thus 
  \eqsplit{
    \sqrt{\frac{nK}{f_k}}\Bigl(\frac{N_k}{n}-\frac{f_k}{K}\pm \frac{1}{K^{1+\al}}\Bigr) \cid \scn(0,1).
   }
   
Apply Lemma \ref{lem:twopoints} to the approximate distribution of $N_k/n - f_k/K$:  
  \eqsplit{
    \uppi(\xi_k=1\mid N_k) &\approx \frac12 + \frac12\sqrt{\frac{n}{f_k K^{1+2\al}}}\Bigl(\frac{N_k}{n}-\frac{f_K}{K}\Bigr).
   }
Which, as expected, shrinks the contribution of the observation to almost nothing. Thus the S-Bayes estimator of $\th$ would be essentially that given by the prior:
 \eqsplit{
    \E \int f^2(x)dx &= \int f_0^2(x)dx + \frac{\pi}{8K^{2\al}}. 
  }
while   
 \eqsplit{
    \int f^2(x)dx &= \int f_0^2(x)dx + \frac{\pi}{8K^{2\al}} + 2\sum_{k=1}^{K}\frac{\xi_k}{K^{\al}}\int_{(k-1)/K}^{k/K}f_0(x)\frac{\pi}{2}\sin\bigl(\pi (Kx-k)\bigr)dx,
  }
  The integrals on the last term are converging to the density, and thus the maximum and minimum are achieved when the $z_k$ are maximally correlated with $f_0$, thus the range of the last term is 
   \eqsplit{
    \pm \frac{2}{K^{\al}}\inf_\eta\int|f_0(x)-\eta|dx.
    }
 Since the only limit on $K$ we have is $K^{1+2\al}\gg n$, then $K^{\al}=L_n n^{\al/(1+2\al)}$, where $L_n$ is any slowly growing function. This concludes the proof.   
\end{proof}

\immediate\write\lastexampfile{%
    \string\setcounter{lastexample}{\thenuexamples}}

\bibliography{onBayesianity1.bib}

@BOOK{Savage1972,
  author =       {L. J. Savage},
  title =        { The Foundations of Statistics (2nd revised ed.)},
  publisher =    {Dover Publications},
  year =         {1972  },
  address =      {New York},
 }

@article{BickelRitov1988,
 ISSN = {0581572X},
 URL = {http://www.jstor.org/stable/25050710},
 author = {P. J. Bickel and Y. Ritov},
 journal = {Sankhyā: The Indian Journal of Statistics, Series A (1961-2002)},
 number = {3},
 pages = {381--393},
 publisher = {Springer},
 title = {Estimating Integrated Squared Density Derivatives: Sharp Best Order of Convergence Estimates},
 urldate = {2025-06-05},
 volume = {50},
 year = {1988}
}

@article{RitovBickel1990,
author = {Y. Ritov and P. J. Bickel},
title = {{Achieving Information Bounds in Non and Semiparametric Models}},
volume = {18},
journal = {The Annals of Statistics},
number = {2},
publisher = {Institute of Mathematical Statistics},
pages = {925 -- 938},
year = {1990},
doi = {10.1214/aos/1176347633},
URL = {https://doi.org/10.1214/aos/1176347633}
}

@article{Tibshirani1996,
 ISSN = {00359246},
 URL = {http://www.jstor.org/stable/2346178},
 author = {Robert Tibshirani},
 journal = {Journal of the Royal Statistical Society. Series B (Methodological)},
 number = {1},
 pages = {267--288},
 publisher = {[Royal Statistical Society, Oxford University Press]},
 title = {Regression Shrinkage and Selection via the Lasso},
 urldate = {2025-06-14},
 volume = {58},
 year = {1996}
}

@book{buhlmann2011statistics,
  title={Statistics for high-dimensional data: methods, theory and applications},
  author={B{\"u}hlmann, Peter and Van De Geer, Sara},
  year={2011},
  publisher={Springer Science \& Business Media},
  address =      {Heidelberg},
}

@article{CastiloEtAl,
author = {Isma{\"e}l Castillo and Johannes Schmidt-Hieber and Aad van der Vaart},
title = {{Bayesian linear regression with sparse priors}},
volume = {43},
journal = {The Annals of Statistics},
number = {5},
publisher = {Institute of Mathematical Statistics},
pages = {1986 -- 2018},
year = {2015},
doi = {10.1214/15-AOS1334},
URL = {https://doi.org/10.1214/15-AOS1334}
}

@article{RigolletTsybakov2012,
author = {Philippe Rigollet and Alexandre B. Tsybakov},
title = {{Sparse Estimation by Exponential Weighting}},
volume = {27},
journal = {Statistical Science},
number = {4},
publisher = {Institute of Mathematical Statistics},
pages = {558 -- 575},
year = {2012},
doi = {10.1214/12-STS393},
URL = {https://doi.org/10.1214/12-STS393}
}

@ARTICLE{Gribonval,
  author={Gribonval, Rémi},
  journal={IEEE Transactions on Signal Processing}, 
  title={Should Penalized Least Squares Regression be Interpreted as Maximum A Posteriori Estimation?}, 
  year={2011},
  volume={59},
  number={5},
  pages={2405-2410},
  doi={10.1109/TSP.2011.2107908}}

@article{EfromovichLow1996,
author = {Sam Efromovich and Mark Low},
title = {{On Bickel and Ritov's conjecture about adaptive estimation of the integral of the square of density derivative}},
volume = {24},
journal = {The Annals of Statistics},
number = {2},
publisher = {Institute of Mathematical Statistics},
pages = {682 -- 686},
year = {1996},
doi = {10.1214/aos/1032894459},
URL = {https://doi.org/10.1214/aos/1032894459}
}

@article{birge1995estimation,
  title={Estimation of integral functionals of a density},
  author={Birg{\'e}, Lucien and Massart, Pascal},
  journal={The Annals of Statistics},
  pages={11--29},
  year={1995},
  publisher={JSTOR}
}

@article{donoho1991geometrizing,
  title={Geometrizing rates of convergence, II},
  author={Donoho, David L and Liu, Richard C},
  journal={The Annals of Statistics},
  pages={633--667},
  year={1991},
  publisher={JSTOR}
}

@article{CaiLow2005,
author = {T. Tony Cai and Mark G. Low},
title = {{Nonquadratic estimators of a quadratic functional}},
volume = {33},
journal = {The Annals of Statistics},
number = {6},
publisher = {Institute of Mathematical Statistics},
pages = {2930 -- 2956},
year = {2005},
doi = {10.1214/009053605000000147},
URL = {https://doi.org/10.1214/009053605000000147}
}

@Inbook{Lindley1990,
author="Lindley, D. V.",
editor="Eatwell, John
and Milgate, Murray
and Newman, Peter",
title="Statistical Inference",
bookTitle="Time Series and Statistics",
year="1990",
publisher="Palgrave Macmillan UK",
address="London",
pages="285--293",
isbn="978-1-349-20865-4",
doi="10.1007/978-1-349-20865-4\_39",
url="https://doi.org/10.1007/978-1-349-20865-4\_39"
}

@article{Lindley1953,
  title={Statistical inference},
  author={Lindley, Dennis V},
  journal={Journal of the Royal Statistical Society Series B: Statistical Methodology},
  volume={15},
  number={1},
  pages={30--65},
  year={1953},
  publisher={Oxford University Press}
}

@article{Lindley2000,
    author = {Lindley, Dennis V.},
    title = {The Philosophy of Statistics},
    journal = {Journal of the Royal Statistical Society Series D: The Statistician},
    volume = {49},
    number = {3},
    pages = {293-337},
    year = {2001},
    month = {12},
    issn = {2515-7884},
    doi = {10.1111/1467-9884.00238},
    url = {https://doi.org/10.1111/1467-9884.00238},
    eprint = {https://academic.oup.com/jrsssd/article-pdf/49/3/293/49931459/jrsssd\_49\_3\_293.pdf},
}

@article{LawRitovInference,
author = {Michael Law and Ya’acov Ritov},
title = {{Inference without compatibility: Using exponential weighting for inference on a parameter of a linear model}},
volume = {27},
journal = {Bernoulli},
number = {3},
publisher = {Bernoulli Society for Mathematical Statistics and Probability},
pages = {1467 -- 1495},
year = {2021},
doi = {10.3150/20-BEJ1280},
URL = {https://doi.org/10.3150/20-BEJ1280}
}

@article{CaiGuo,
 author = {Cai, T. Tony and Guo, Zijian},
 title = {Semisupervised inference for explained variance in high dimensional linear regression and its applications},
 fjournal = {Journal of the Royal Statistical Society. Series B. Statistical Methodology},
 journal = {J. R. Stat. Soc., Ser. B, Stat. Methodol.},
 issn = {1369-7412},
 volume = {82},
 number = {2},
 pages = {391--419},
 year = {2020},
 language = {English},
 doi = {10.1111/rssb.12357},
 zbMATH = {7554759}
}

@article{JavanmardMontanari,
 author = {Javanmard, Adel and Montanari, Andrea},
 title = {Confidence intervals and hypothesis testing for high-dimensional regression},
 fjournal = {Journal of Machine Learning Research (JMLR)},
 journal = {J. Mach. Learn. Res.},
 issn = {1532-4435},
 volume = {15},
 pages = {2869--2909},
 year = {2014},
 language = {English},
 url = {jmlr.csail.mit.edu/papers/v15/javanmard14a.html},
 zbMATH = {6378156},
 Zbl = {1319.62145}
}

@article{Zuerich,
 author = {van de Geer, Sara and B{\"u}hlmann, Peter and Ritov, Ya'acov and Dezeure, Ruben},
 title = {On asymptotically optimal confidence regions and tests for high-dimensional models},
 fjournal = {The Annals of Statistics},
 journal = {Ann. Stat.},
 issn = {0090-5364},
 volume = {42},
 number = {3},
 pages = {1166--1202},
 year = {2014},
 language = {English},
 doi = {10.1214/14-AOS1221},
 zbMATH = {6324583},
 Zbl = {1305.62259}
}

@article{ZhangZhang,
    author = {Zhang, Cun-Hui and Zhang, Stephanie S.},
    title = {Confidence Intervals for Low Dimensional Parameters in High Dimensional Linear Models},
    journal = {Journal of the Royal Statistical Society Series B: Statistical Methodology},
    volume = {76},
    number = {1},
    pages = {217-242},
    year = {2013},
    month = {07},
    issn = {1369-7412},
    doi = {10.1111/rssb.12026},
    url = {https://doi.org/10.1111/rssb.12026},
    eprint = {https://academic.oup.com/jrsssb/article-pdf/76/1/217/49514350/jrsssb\_76\_1\_217.pdf},
}

@Inbook{Bernardo2011,
author="Bernardo, Jos{\'e} M.",
editor="Lovric, Miodrag",
title="Bayesian Statistics",
bookTitle="International Encyclopedia of Statistical Science",
year="2011",
publisher="Springer Berlin Heidelberg",
address="Berlin, Heidelberg",
pages="107--133",
isbn="978-3-642-04898-2",
doi="10.1007/978-3-642-04898-2\_139",
url="https://doi.org/10.1007/978-3-642-04898-2\_139"
}

@article{RBGK,
author = {Y. Ritov and P. J. Bickel and A. C. Gamst and B. J. K. Kleijn},
title = {{The Bayesian Analysis of Complex, High-Dimensional Models: Can It Be CODA?}},
volume = {29},
journal = {Statistical Science},
number = {4},
publisher = {Institute of Mathematical Statistics},
pages = {619 -- 639},
year = {2014},
doi = {10.1214/14-STS483},
URL = {https://doi.org/10.1214/14-STS483}
}

@article{GreenshteinRitov2004,
 ISSN = {13507265},
 URL = {http://www.jstor.org/stable/3318880},
 author = {Eitan Greenshtein and Ya'Acov Ritov},
 journal = {Bernoulli},
 number = {6},
 pages = {971--988},
 publisher = {International Statistical Institute (ISI) and Bernoulli Society for Mathematical Statistics and Probability},
 title = {Persistence in High-Dimensional Linear Predictor Selection and the Virtue of Overparametrization},
 urldate = {2025-05-15},
 volume = {10},
 year = {2004}
}

@article{efron2019bayes,
  title={Bayes, oracle Bayes and empirical Bayes},
  author={Efron, Bradley},
  journal={Statistical science},
  volume={34},
  number={2},
  pages={177--201},
  year={2019},
  publisher={JSTOR}
}

@article{ritov2024no,
  title={No need for an oracle: the nonparametric maximum likelihood decision in the compound decision problem is minimax},
  author={Ritov, Ya’acov},
  journal={Statistical Science},
  volume={39},
  number={4},
  pages={637--643},
  year={2024},
  publisher={Institute of Mathematical Statistics}
}

@article{CODA,
author = {Robins, James M. And Ritov, Ya'acov},
title = {TOWARD A CURSE OF DIMENSIONALITY APPROPRIATE (CODA) ASYMPTOTIC THEORY FOR SEMI-PARAMETRIC MODELS},
journal = {Statistics in Medicine},
volume = {16},
number = {3},
pages = {285-319},
doi = {https://doi.org/10.1002/(SICI)1097-0258(19970215)16:3<285::AID-SIM535>3.0.CO;2-\#},
url = {https://onlinelibrary.wiley.com/doi/abs/10.1002/\%28SICI\%291097-0258\%2819970215\%2916\%3A3\%3C285\%3A\%3AAID-SIM535\%3E3.0.CO\%3B2-\%23},
eprint = {https://onlinelibrary.wiley.com/doi/pdf/10.1002/\%28SICI%291097-0258\%2819970215\%2916%3A3\%3C285\%3A\%3AAID-SIM535\%3E3.0.CO\%3B2-\%23},
year = {1997}
}

@article{HT1952,
 ISSN = {01621459, 1537274X},
 URL = {http://www.jstor.org/stable/2280784},
 author = {D. G. Horvitz and D. J. Thompson},
 journal = {Journal of the American Statistical Association},
 number = {260},
 pages = {663--685},
 publisher = {[American Statistical Association, Taylor & Francis, Ltd.]},
 title = {A Generalization of Sampling Without Replacement From a Finite Universe},
 urldate = {2025-05-24},
 volume = {47},
 year = {1952}
}

@ARTICLE{SahaGunt20,
  author =       {Saha, Sujayam  and  Guntuboyina, Adityanand},
  title =        {On the nonparametric maximum likelihood estimator for
gaussian location mixture densities with application to
Gaussian denoising},
  journal =      {Annals of Statistics},
  year =         {2020},
  volume =       {48},
  pages =        {738--762},
}

@article{GreenshteinRitov22,
  title={Generalized maximum likelihood estimation of the mean of parameters of mixtures. With applications to sampling and to observational studies},
  author={Greenshtein, Eitan and Ritov, Ya’acov},
  journal={Electronic Journal of Statistics},
  volume={16},
  number={2},
  pages={5934--5954},
  year={2022},
  publisher={The Institute of Mathematical Statistics and the Bernoulli Society}
}

@article{Mayo2014,
author = {Deborah G. Mayo},
title = {{On the Birnbaum Argument for the Strong Likelihood Principle}},
volume = {29},
journal = {Statistical Science},
number = {2},
publisher = {Institute of Mathematical Statistics},
pages = {227 -- 239},
year = {2014},
doi = {10.1214/13-STS457},
URL = {https://doi.org/10.1214/13-STS457}
}

@article{birnbaum1962foundations,
  title={On the foundations of statistical inference},
  author={Birnbaum, Allan},
  journal={Journal of the American Statistical Association},
  volume={57},
  number={298},
  pages={269--306},
  year={1962},
  publisher={Taylor \& Francis}
}

@inproceedings{berger1988likelihood,
  title={The likelihood principle},
  author={Berger, James O and Wolpert, Robert L},
  year={1988},
  organization={IMS}
}

@incollection{wasserman1998asymptotic,
  title={Asymptotic properties of nonparametric Bayesian procedures},
  author={Wasserman, Larry},
  booktitle={Practical nonparametric and semiparametric Bayesian statistics},
  pages={293--304},
  year={1998},
  publisher={Springer}
}

@article{harmeling2007bayesian,
  title={Bayesian estimators for Robins-Ritov's problem},
  author={Harmeling, Stefan and Touissant, M},
  year={2007},
  publisher={School of Informatics, University of Edinburgh}
}

@misc{datta2025inverseprobabilityweightingsurvey,
      title={Inverse Probability Weighting: from Survey Sampling to Evidence Estimation}, 
      author={Jyotishka Datta and Nicholas Polson},
      year={2025},
      eprint={2204.14121},
      archivePrefix={arXiv},
      primaryClass={stat.ME},
      url={https://arxiv.org/abs/2204.14121}, 
}

@article{li2010bayesian,
  title={Are Bayesian Inferences Weak for Wasserman's Example?},
  author={Li, Longhai},
  journal={Communications in Statistics—Simulation and Computation{\textregistered}},
  volume={39},
  number={3},
  pages={655--667},
  year={2010},
  publisher={Taylor \& Francis}
}

@book{wasserman2004all,
  title={All of statistics: a concise course in statistical inference},
  author={Wasserman, Larry},
  year={2004},
  publisher={Springer Science \& Business Media}
}

@article{Bayes1763,
 ISSN = {02607085},
 URL = {http://www.jstor.org/stable/105741},
 author = {Mr. Bayes and Mr. Price},
 journal = {Philosophical Transactions (1683-1775)},
 pages = {370--418},
 publisher = {The Royal Society},
 title = {An Essay towards Solving a Problem in the Doctrine of Chances. By the Late Rev. Mr. Bayes, F. R. S. Communicated by Mr. Price, in a Letter to John Canton, A. M. F. R. S.},
 urldate = {2025-07-07},
 volume = {53},
 year = {1763}
}

@article{GILLIES1987,
title = {Was Bayes a Bayesian?},
journal = {Historia Mathematica},
volume = {14},
number = {4},
pages = {325-346},
year = {1987},
issn = {0315-0860},
doi = {https://doi.org/10.1016/0315-0860(87)90065-6},
url = {https://www.sciencedirect.com/science/article/pii/0315086087900656},
author = {Donald A Gillies}
}

@article{Stigler1982,
author = {Stigler, Stephen M.},
title = {Thomas Bayes's Bayesian Inference},
journal = {Journal of the Royal Statistical Society: Series A (General)},
volume = {145},
number = {2},
pages = {250-258},
doi = {https://doi.org/10.2307/2981538},
url = {https://rss.onlinelibrary.wiley.com/doi/abs/10.2307/2981538},
eprint = {https://rss.onlinelibrary.wiley.com/doi/pdf/10.2307/2981538},
year = {1982}
}

@book{Bayes1736,
  title={An Introduction to the Doctrine of Fluxions: And Defence of the Mathematicans Against the Objections of the Author of the Analyst},
  author={Bayes, Thomas},
  year={1736},
  publisher={J. Noon}
}

@article{berger2000bayesian,
  title={Bayesian analysis: A look at today and thoughts of tomorrow},
  author={Berger, James O},
  journal={Journal of the American Statistical Association},
  volume={95},
  number={452},
  pages={1269--1276},
  year={2000},
  publisher={Taylor \& Francis}
}

@book{berger2013statistical,
  title={Statistical decision theory and Bayesian analysis},
  author={Berger, James O},
  year={2013},
  publisher={Springer Science \& Business Media}
}

@article{Ruberg29032019,
author = {Stephen J. Ruberg and Frank E. Harrell Jr. and Margaret Gamalo-Siebers and Lisa LaVange and J. Jack Lee and Karen Price and Carl Peck},
title = {Inference and Decision Making for 21st-Century Drug Development and Approval},
journal = {The American Statistician},
volume = {73},
number = {sup1},
pages = {319--327},
year = {2019},
publisher = {ASA Website},
doi = {10.1080/00031305.2019.1566091}}

@article{Wasserstein29032019,
author = {Ronald L. Wasserstein and Allen L. Schirm and Nicole A. Lazar},
title = {Moving to a World Beyond ``p < 0.05''},
journal = {The American Statistician},
volume = {73},
number = {sup1},
pages = {1--19},
year = {2019},
publisher = {ASA Website},
doi = {10.1080/00031305.2019.1583913}}

@article{Samuel1963,
 ISSN = {00034851},
 URL = {http://www.jstor.org/stable/2238345},
 author = {Ester Samuel},
 journal = {The Annals of Mathematical Statistics},
 number = {4},
 pages = {1370--1385},
 publisher = {Institute of Mathematical Statistics},
 title = {An Empirical Bayes Approach to the Testing of Certain Parametric Hypotheses},
 urldate = {2025-07-08},
 volume = {34},
 year = {1963}
}

@article{IHmodel,
author = {Bickel, P. J. and Ritov, Y.},
title = {Efficient Estimation Using Both Direct and Indirect Observations},
journal = {Theory of Probability \& Its Applications},
volume = {38},
number = {2},
pages = {194-213},
year = {1994},
doi = {10.1137/1138022},
URL = {        https://doi.org/10.1137/1138022},
eprint = {https://doi.org/10.1137/1138022}
}
\end{document}